\documentclass[preprint,12pt]{article}

\usepackage{float}
\usepackage{amssymb}
\usepackage{mathrsfs}
\usepackage{enumerate}
\usepackage{subfigure}
\usepackage{graphicx}
\usepackage{colortbl,dcolumn}
\usepackage{tabularx}
\usepackage{amsmath}
\usepackage{psfrag}
\usepackage{booktabs}
\usepackage{cite}
\usepackage{amsthm}

\usepackage[top=1.2in, bottom=1.2in, left=1in, right=1in]{geometry}

\usepackage{subfigure}

\newcommand{\R}{\mathbb{R}}
\newcommand{\N}{\mathbb{N}}
\newcommand{\E}{\mathbb{E}}
\newcommand{\F}{\mathbb{F}}
\newcommand{\G}{\mathbb{G}}
\renewcommand{\H}{\mathbb H}
\renewcommand{\P}{\mathbb{P}}

\newcommand{\B}{\mathbb{B}}

\newcommand{\Z}{\mathbb{Z}}

\newcommand{\tr}{\operatorname{trace}}

\newcommand{\Q}{\mathbf Q}

\renewcommand{\AA}{\mathbb A}

\newcommand{\DD}{\mathcal D}
\newcommand{\GG}{\mathcal G}
\newcommand{\LL}{\mathcal L}
\newcommand{\OO}{\mathcal O}

\renewcommand{\P}{\mathbb P}

\newcommand{\MM}{\mathbb M}

\newcommand{\FFF}{\mathcal F}
\newcommand{\OOO}{\Omega}
\newcommand{\<}{\langle}
\renewcommand{\>}{\rangle}

\newcounter{RomanNumber}
\newcommand{\MyRoman}[1]{\setcounter{RomanNumber}{#1}
{\rm\Roman{RomanNumber}}}
\renewcommand{\tr}{{\rm Tr}}

\newtheorem{theorem}{Theorem}[section]

\newtheorem{lm}{Lemma}[section]
\newtheorem{prop}{Proposition}[section]
\newtheorem{cor}{Corollary}[section]
\newtheorem{rk}{Remark}[section]

\begin{document}
\title{Energy-preserving exponential integrable numerical method for stochastic cubic wave equation with additive noise}

       \author{
       Jianbo Cui\footnotemark[1],
        Jialin Hong\footnotemark[2],
         Lihai Ji\footnotemark[3]
         and
        Liying Sun\footnotemark[4],\\
       {\small
       \footnotemark[1]~School of Mathematics, Georgia Institute of Technology, Atlanta, GA 30332, USA.}\\
        {\small\footnotemark[2]~\footnotemark[4]~Institute of Computational Mathematics and Scientific/Engineering Computing,}
        \\{\small Academy of Mathematics and Systems Science, Chinese Academy of Sciences, }
        {\small Beijing 100190, P.R.China }\\
        {\small\footnotemark[3]~Institute of Applied Physics  and Computational Mathematics, Beijing, 100094, China.}
        }
       \maketitle
       \footnotetext{\footnotemark[2]\footnotemark[4] Authors are supported by National Natural Science Foundation of China (NO. 91130003, NO. 11021101 and NO. 11290142).}
       \footnotetext{\footnotemark[3]The author is funded by National Natural Science Foundation of China (No. 11601032)}
        \footnotetext{\footnotemark[4]Corresponding author: liyingsun@lsec.cc.ac.cn}

\maketitle

       \begin{abstract}
          {\rm\small In this paper, we present an energy-preserving exponentially integrable numerical method for stochastic wave equation with cubic nonlinearity and additive noise.
          	We first apply the spectral Galerkin method to discretizing the original equation and show that this spatial discretization possesses an energy evolution law and certain exponential integrability property.
          	Then the exponential integrability property of the exact solution is deduced by proving the strong convergence of the semi-discretization.
          	To propose a full discrete numerical method which could inherit both the energy evolution law and the exponential integrability, we use the splitting technique and averaged vector field method in the temporal direction. 
          	Combining these structure-preserving properties with regularity estimates of the exact and the numerical solutions, we obtain the strong convergence rate of the numerical method.
          	Numerical experiments coincide with these theoretical results.}\\


\end{abstract}

\section{Introduction}
\label{int}
As a kind of commonly observed physical phenomena, the wave motions are usually described by stochastic partial differential equations (SPDEs) of hyperbolic type.
The main objective of this paper is to numerically investigate the following stochastic wave equation with cubic nonlinearity, driven by an additive noise
\begin{equation}
\left\{
\begin{aligned}
\label{mod;swe}
&du(t)=v(t)dt,\quad\quad \;\; &\text{in} \;\; \OO \times (0,T],\\
&dv(t)=\Lambda u(t)dt-f(u(t))dt+dW(t), \quad\quad \;\; &\text{in} \;\; \OO \times (0,T], \\
&u(0)=u_0, \quad v(0)=v_0, \quad \;\; &\text{in} \;\; \OO,
\end{aligned}
\right.
\end{equation}
where $\OO=(0,1)^{d}$ with $d\le 2$, $T\in(0,\infty)$ and $u_0,v_0:\OO\rightarrow \R$ are deterministic.
Assume that $\Lambda=\sum\limits_{i=1}^d\frac{\partial^2}{\partial x_i^2}$ is the Laplace operator with homogeneous Dirichlet boundary condition, and the nonlinear term $f(u)=c_3u^3+\cdots+c_1u+c_0$ is assumed to be a polynomial 
with $c_3>0$.
Throughout this paper, $W$ is an $L^2:=L^2(\OO;\R)$-valued $\Q$-Wiener process with respect to a filtered probability space $(\Omega, \FFF, \{\FFF_t\}_{t\geq 0}, \P)$, i.e., there exists an orthonormal basis $\{e_k\}_{k\in \N^+} $ of $L^2$ and a sequence of mutually independent real-valued Brownian motions $\{\beta_k\}_{k\in \N^+} $ such that $W(x,t)=\sum_{k\in \N^+} \Q^{\frac12}e_k(x)\beta_k(t)$
with $\Q$ being a symmetric, positive definite and finite trace operator.
For the well-posedness of stochastic wave equation, we refer to 
\cite{C02,Chow06} for the existence and uniqueness of the mild solution with more general polynomial drift coefficients, and to \cite{DS09} with more general driving noises. 
The evolution of the energy, as an intrinsic quantity of the wave equation, is studied in \cite{C02}.
As another important property (see e.g., \cite{AMS94}), the exponential integrability property of the solution of the stochastic wave equation has not been well understood. 
We are only aware that the authors in \cite{CHJ13}
prove the exponential integrability property of the spectral Galerkin approximation of $2$-dimensional  stochastic wave equation.
Utilizing the uniform exponential integrability property and regularity estimate of the spectral Galerkin method applied to \eqref{mod;swe}, in this paper we prove that the exact solution  admits the following exponential integrability property
\begin{align*}
\E\left(\exp\left(\int_0^T c\|u(s)\|_{L^6}^2ds\right)\right)\le C,
\end{align*}
where $c\in \mathbb R$, $T\in (0,\infty)$, $C:=C(u_0,v_0,\Q,T,c,d)$ and $d\leq 2.$

Finding solutions numerically of stochastic wave equation  is an active ongoing research area. For instance,
the authors in 
\cite{WGT14} present higher order strong approximations consisting of the Galerkin approximation in space combined with the trigonometric time integrator  for stochastic wave equation with regular and Lipschitz coefficients driven by additive space-time white noise.
The authors in 
\cite{ACLW16} obtain the strong convergence rate of a full discrete scheme for stochastic wave equations with regular and Lipschitz coefficients, and they provide an almost trace formula for their proposed full discretization.
With regard to the weak convergence analysis, we refer to e.g.,  
\cite{KLL13} for  finite element approximations of linear stochastic wave equation with additive noise, and to
\cite{JJT15} for spatial spectral Galerkin approximations of stochastic wave equation with multiplicative noise.
For the stochastic wave equation with cubic nonlinearity,
we are only aware that  some nonstandard partial-implicit midpoint-type difference method which can control its energy functional in a dynamically consistent fashion is proposed in \cite{S08} with $d=1$.

From the numerical viewpoint, it is natural and important to design numerical methods to inherit both the energy evolution law and exponential integrability property of the original system. 
For instance,
with regard to the energy-preserving numerical method,
we refer to e.g., \cite{CH17} for the time splitting method of the stochastic nonlinear Schr\"odinger equation and  to \cite{CLS13} for the exponential integrator of the stochastic linear wave equation. 
With regard to the exponentially integrable numerical method, we 
refer to e.g., \cite{BCH18} for an explicit splitting scheme of the stochastic parabolic equation,   to \cite{HJ14,HJW18}
for some stopped increment-tamed Euler approximations of stochastic differential equations, and to 
\cite{CHL16b,CHLZ17b} for the finite difference method and the splitting full discrete scheme of the stochastic nonlinear Schr\"odinger equation.
However, up to now, there has been no result on  full discretizations which could inherit both the energy evolution law and the exponential integrability for stochastic nonlinear wave equation with non-globally Lipschitz coefficients.

In the present paper,  we propose a strategy to design 
numerical methods  preserving both the energy evolution law and exponential integrability property. The idea is  based on splitting the original equation in temporal direction into a deterministic Hamiltonian system and a stochastic system.
We combine the splitting technique with the averaged vector field (AVF) method, and apply  spectral Galerkin method in spatial direction to present the numerical method
\begin{equation*}
\begin{split}
u^N_{m+1}=&u^N_{m}+h\frac{v^N_{m}+\bar v^N_{m+1}}{2},\\
\bar v^N_{m+1}=&v^N_m+h\Lambda_N \frac{u^N_{m}+ u^N_{m+1}}{2}-hP_N\left(\int_0^1f(u^N_m+\theta(u^N_{m+1}-u^N_m))d\theta \right),\\
v^N_{m+1}=&\bar v^N_{m+1}+P_N\delta W_{m},
\end{split}
\end{equation*}
where $N\in \mathbb N^+=\{1,2,\cdots\}$, $h=T/M$ with $M\in \N^+$ is the time step-size, $m\in \Z_M:=\{0,1,\cdots,M-1\}.$
Here $P_N$ is the spectral Galerkin projection operator defined in \eqref{def:pjp}, and $\delta W_{m}$ is the increment of the Wiener process defined in \eqref{sche;sbes}. 
The averaged vector field method (AVF) is one kind of the discrete gradient approach to construct numerical schemes with conservation properties, which has been discussed in the deterministic wave equation (see \cite{Fur01}). 
We prove that the proposed numerical method preserves the energy evolution law and the  exponential integrability property
\begin{align*}
\E
\left(\exp\left(ch\sum\limits_{i=1}^M\|u^N_i\|_{L^6}^2
\right)
\right)
\leq C,
\end{align*}
where $c\in \mathbb R$, $C:=C(u_0,v_0,\Q,T,c,d)$ and $d\leq 2.$
Based on the above structure-preserving properties, we show that 
the strong convergence orders of the numerical method are $\beta/2$ in spatial direction and $\min(\beta,2)/2$ in temporal direction for the case that $d=1$, $\beta\ge 1$ or that $d= 2$, $\beta= 2$,  where $\beta$ is the regularity index of the $\mathbf Q$-Wiener process $W.$ 
To the best of our knowledge, this is the first result regarding both the exponential integrability and the strong convergence rate of full discretizations for stochastic wave equations with cubic nonlinearity.

The rest of this paper is organized as follows.
Section \ref{sec2} presents an abstract formulation of the stochastic wave equation, and introduces some properties of the corresponding group.
In Section \ref{sec3}, the regularity estimate and exponential integrability property of the mild solution of the spectral Galerkin discretization are studied.
The analysis of strong convergence for the spectral Galerkin discretization is also presented.
Section \ref{sec4} is devoted to constructing the numerical method which preserves the energy evolution law and exponential integrability property, and deducing its $L^p(\OOO;\H)$ error estimate.
Numerical experiments are carried out in Section \ref{sec5} to verify theoretical results.

\section{Preliminary and frame work}
\label{sec2}
In this section, we first set forth an abstract formulation of \eqref{mod;swe} for the stochastic wave equation, and introduce some properties of the unitary group generated by the dominant operator.
Throughout this paper, the constant $C$ may be different from line to line but never depending on $N$ and $h$.

Assume that the eigenvalues $0< \lambda_1\leq\lambda_2\leq\cdots$ and that  the corresponding eigenfunctions $\{e_i\}_{i=1}^\infty$ of the operator $-\Lambda$, i.e., with
$-\Lambda e_i=\lambda_ie_i,$ $i\in\mathbb N^{+}$, form an orthonormal basis in $L^2$.
Define the interpolation space $\dot\H^r:=\DD((-\Lambda)^\frac r2)$ for $r\in \R$  equipped with the inner product
$
\<x,y\>_{\dot{\H}^r}=\left\<(-\Lambda)^\frac r2x,(-\Lambda)^\frac r2y\right\>_{L^2}
=\sum\limits_{i=1}^\infty\lambda_i^r\<x,e_i\>_{L^2}\<y,e_i\>_{L^2}
$
and the corresponding norm $\|x\|_{\dot{\H}^r}:=\<x,x\>_{\dot{\H}^r}^{1/2}$.
Furthermore, we introduce the product space
$
\H^r:=\dot\H^r\times \dot\H^{r-1},$ $r\in \R,
$
endowed with the inner product
$
\<X_1,X_2\>_{\H^r}=\<x_1,x_2\>_{\dot\H^r}+\<y_1,y_2\>_{\dot\H^{r-1}}
$
for any $X_1=(x_1,y_1)^\top$ and $X_2=(x_2,y_2)^\top,$ and the corresponding norm 
$
\|X\|_{\H^r}:=\<X,X\>_{\H^r}^{1/2}=\big(\|x\|_{\dot{\mathbb H}^r}^2+\|y\|_{\dot{\mathbb H}^{r-1}}^2\big)^{1/2} $ for $X=(x,y)^\top.
$

Given two separable Hilbert spaces $(\mathcal H, \|\cdot \|_{\mathcal H})$ and $(\widetilde  H,\|\cdot \|_{\widetilde H})$, 
$\LL(\mathcal H, \widetilde H)$  and $\LL_1(\mathcal H, \widetilde H)$ are the Banach spaces of all linear bounded operators 
and  the nuclear operators from $\mathcal H$ to $\widetilde H$, respectively. 
The trace of an operator $\mathcal T\in \LL_1(\mathcal H):= \LL_1(\mathcal H,\mathcal H)$
is $\tr[\mathcal T]=\sum_{k\in \N^+}\<\mathcal Tf_k,f_k\>_{\mathcal H}$, where $\{f_k\}_{k\in \N^+}$ is any orthonormal basis of $\mathcal H$.
In particular, if $\mathcal T\ge 0$, $\text{Tr}(\mathcal T)=\|\mathcal T\|_{\mathcal L_1( \mathcal H)}$.
Denote by $\LL_2(\mathcal H,\widetilde H)$ the space 
of Hilbert--Schmidt operators from $\mathcal H$ into $\widetilde H$, equipped with  the norm $\|\cdot\|_{\LL_2(\mathcal H,\widetilde H)}=(\sum_{k\in \N^+}\|\cdot f_k\|^2_{\widetilde H})^{\frac{1}{2}}$.
For convenience, we denote $\LL_2(\mathcal H):=\LL_2(\mathcal H,\mathcal H)$
and $L^p:=L^p(\mathcal O,\mathbb R), p\ge 1.$

Denote $X=(u,v)^\top$.  The abstract form of \eqref{mod;swe} is
\begin{equation}
\begin{split}
\label{mod;swe1}
dX(t)&=AX(t)dt+\F(X(t))dt+\G dW(t),\quad t\in(0,T],\\
X(0)&=X_0,
\end{split}
\end{equation}
where
\begin{align*}
X_0=
\begin{bmatrix}
u_0\\
v_0
\end{bmatrix},\quad
A=
\begin{bmatrix}
0 & I\\
\Lambda & 0
\end{bmatrix},\quad
\F(X(t))=
\begin{bmatrix}
0\\
-f(u(t))
\end{bmatrix},\quad
\G=
\begin{bmatrix}
0\\
I
\end{bmatrix}.
\end{align*}
Here and below we denote $I$ by the identity operator defined in  $L^2$. 
Moreover, we define the domain of operator $A$ by
\begin{align*}
\DD(A)=\left\{X\in\H:
AX=\begin{bmatrix}
v\\
\Lambda u
\end{bmatrix}
\in\H:=L^2\times \dot\H^{-1}\right\}=\dot\H^1\times L^2,
\end{align*}
then the operator $A$ generates a unitary group $E(t),$ $t\in \R,$ on $\H$, given by
\begin{align*}
E(t)=\exp (tA)=
\begin{bmatrix}
C(t)& (-\Lambda)^{-\frac 12}S(t)\\
-(-\Lambda)^{\frac 12}S(t) & C(t)
\end{bmatrix},
\end{align*}
where $C(t)=\cos(t(-\Lambda)^{\frac 12})$ and $S(t)=\sin(t(-\Lambda)^{\frac 12})$ are the cosine and sine operators, respectively.

Throughout this article, we will consider the mild solution of \eqref{mod;swe1}, that is,
\begin{align}
\label{sol;swe}
X(t)=E(t)X_0+\int_0^tE(t-s)\mathbb F(X(s))ds+\int_0^t E(t-s)\mathbb G dW(s),\quad a.s.
\end{align}
We refer to \cite{C02,Chow06} for the well-posedness of the mild solution for the stochastic wave equation. 
The following lemma concerns with the temporal H\"{o}lder continuity of both sine and cosine operators, which has been discussed, for example, in \cite{ACLW16}.

\begin{lm}
	\label{lm;prpE_n}
	For $r\in[0,1],$ there exists a positive constant $C':=C'(r)$ such that
	\begin{align*}
	&\|(S(t)-S(s))(-\Lambda)^{-\frac r2}\|_{\LL(L^2)}\leq C'(t-s)^r,\\
	&\|(C(t)-C(s))(-\Lambda)^{-\frac r2}\|_{\LL(L^2)}\leq C'(t-s)^r
	\end{align*}
	and
	$
	\|(E(t)-E(s))X\|\leq C'(t-s)^r\|X\|_{\H^r}
	$ 
	for all $t\geq s\geq 0.$
\end{lm}

\begin{lm}
	\label{lm;2}
	For any $t\in\R,$ 
	$C(t)$ and $S(t)$ satisfy a trigonometric identity in the sense that
	$
	\|S(t)x\|_{L^2}^2+\|C(t)x\|_{L^2}^2=\|x\|_{L^2}^2$ for $x\in L^2.$
\end{lm}

Based on the above trigonometric identity, it  can be obtained that $\|E(t)\|_{\LL(\H)}=1,$  $t\in\R.$ 
Denote the potential functional by $F$ such that $\frac {\partial F}{\partial u}(u)=f(u)$. 
Then the functional $F: L^4 \to \mathbb R$ can be chosen such that 
\begin{align}\label{equ}
a_1\|u\|_{L^4}^4-b_1 \le F(u) \le a_2\|u\|_{L^4}^4+b_2
\end{align} 
for some positive constants $a_1,a_2,b_1,b_2$.
Similar to \cite{C02}, using finite  dimensional approximation and It\^o's formula to the Lyapunov energy functional $V_1:\mathbb{H}^1\rightarrow \mathbb{R}$
\begin{align}\label{Lya-V1}
V_1(u,v)=\frac 12\|u\|^2_{\dot\H^1}+\frac 12\|v\|^2_{L^2}
+F(u)+C_1,\quad C_1\ge b_1,
\end{align}
we have the following energy evolution law of \eqref{mod;swe} by taking the limit.

\begin{lm}
	\label{lm;eel}
	Assume that $X_0\in \H^1$ and $\Q^{\frac 12}\in \mathcal L_2(L^2).$
	Then the stochastic wave equation  { \eqref{mod;swe}} admits the energy evolution law
	\begin{align*}
	\E(V_1(u(t),v(t)))
	=&V_1(u_0, v_0)
	+\frac 12{\rm{Tr}}\left(\Q\right)t, \quad t\in \mathbb R^+.
	\end{align*}
\end{lm}

\section{Exponential integrability property of stochastic wave equation}
\label{sec3}
This section is devoted to analyzing the spatial spectral Galerkin method of  \eqref{mod;swe1}. 
We present the existence, uniqueness and regularity estimate for the solution to the spectral Galerkin discretization, including the uniform boundedness of the solution in $L^{p}(\Omega;\mathcal C([0,T];\mathbb H^{\beta}))$-norm and H\"{o}lder continuity of the solution in $L^p(\OOO;\H)$-norm. 
By giving the strong convergence of the spectral Galerkin method, we show the exponential integrability property of the exact solution of \eqref{mod;swe}.

\subsection{Spectral Galerkin method}
In this subsection, we study the spectral Galerkin method for stochastic wave equation \eqref{mod;swe1}.
The spectral Galerkin method has been used to discretize SPDEs in spatial direction (see e.g., \cite{CHLZ17b,JJT15} and references therein).
For the considered equation and $N\in\N^+,$ we define a finite dimensional subspace $U_N$ of $L^2$ spanned by $\{e_1,e_2,\cdots,e_N\}$,
and the projection operator
$P_N:\dot\H^r\rightarrow U_N$ by
\begin{align}
\label{def:pjp}
P_N\zeta=\sum\limits_{i=1}^N\<\zeta,e_i\>_{L^2}e_i,\quad \forall ~\zeta\in \dot\H^r,\quad r\geq -1,
\end{align}
which satisfies  $
\|P_N\|_{\mathcal L(L^2)}\leq 1.$ 
Define $\Lambda_N:U_N\rightarrow U_N$ by
\begin{align}\label{la_operator}
\Lambda_N\zeta
=\Lambda P_N\zeta=P_N\Lambda\zeta=-\sum\limits_{i=1}^N\lambda_i\<\zeta,e_i\>_{L^2}e_i,\quad \forall~ \zeta\in U_N.
\end{align}
By denoting $X^N=(u^N,v^N)^\top$, the spectral Galerkin method applied to \eqref{mod;swe1} yields
\begin{equation}
\begin{split}
\label{mod;sga}
dX^N(t)&=A_NX^N(t)dt+\F_N(X^N(t))dt+\G_N dW(t),\quad t\in(0,T],\\
X^N(0)&=X_0^N,
\end{split}
\end{equation}
where
\begin{align*}
X^N_0=
\begin{bmatrix}
u_0^N\\
v_0^N
\end{bmatrix},\quad
A_N=
\begin{bmatrix}
0 & I\\
\Lambda_N & 0
\end{bmatrix},\quad
\F_N(X^N)=
\begin{bmatrix}
0\\
-P_N\left(f(u^N)\right)
\end{bmatrix},\quad
\G_N=
\begin{bmatrix}
0\\
P_N
\end{bmatrix}
\end{align*}
with  $u_0^N=P_Nu_0,v_0^N=P_Nv_0.$ 
Similarly, the discrete operator $A_N$ 	generates a unitary group
\begin{align*}
E_N(t)=\exp (tA_N)=
\begin{bmatrix}
C_N(t)& (-\Lambda_N)^{-\frac 12}S_N(t)\\
-(-\Lambda_N)^{\frac 12}S_N(t) & C_N(t)
\end{bmatrix},\quad t\in \R,
\end{align*}
where $C_N(t)=\cos(t(-\Lambda_N)^{\frac 12})$ and $S_N(t)=\sin(t(-\Lambda_N)^{\frac 12})$ are the discrete cosine and sine operators defined in $U_N,$ respectively.
It can be verified straightforwardly that
\begin{align*}
C_N(t)P_N\zeta=C(t)P_N\zeta=P_NC(t)\zeta,\quad
S_N(t)P_N\zeta=S(t)P_N\zeta=P_NS(t)\zeta
\end{align*}
for any $\zeta\in\dot\H^r,$ $r\geq -1.$

Thanks to the Lyapunov function $V_1$ in \eqref{Lya-V1},  we are able to show the existence and uniqueness of the mild solution of \eqref{mod;sga} and a priori estimation on $X^N(t)$.
Since the coefficient in \eqref{mod;sga} is locally Lipschitz and $A_N$ is a bounded operator in $U_N\times U_N$, the local existence of the unique mild solution is obtained by using the Banach fixed-point theorem or Picard iterations
under the $\mathcal C([0,T];L^p(\Omega;\H^1))$-norm for $p\ge2.$
To extend the local solution to a global solution, similar to  the proof of \cite[Theorem 4.2]{C02}, we show a priori estimate 
in $L^{p}(\Omega;\mathcal C([0,T];\H^1))$ by applying It\^o's formula to $\left(V_1(u^N,v^N)\right)^p,$ $p\in [2,\infty)$ and using the unitary property of $E_N(t)$. Finally, we obtain the following properties of $X^N$.

\begin{lm}\label{h1}
	Assume that $X_0\in \H^1$, $T>0$ and $\Q^{\frac 12}\in \mathcal L_2(L^2).$
	Then the spectral Galerkin discretization \eqref{mod;sga} has a unique mild solution given by
	\begin{align}
	\label{sol;sga}
	X^N(t)=E_N(t)X^N_0+\int_0^tE_N(t-s)\F_N(X^N(s))ds+
	\int_0^tE_N(t-s)\G_NdW(s)
	\end{align}
	for $t\in[0,T].$
	Moreover, for $p\ge 2$, there exists a positive constant $C:=C(X_0,T,\Q,p)$ such that
	\begin{align}
	\label{spec;H1}
	\|X^N\|_{L^{p}(\Omega;\mathcal C([0,T];\H^1))}\leq
	C.
	\end{align}
\end{lm}

\begin{prop}
	Assume that $X_0\in \H^1$ and $\Q^{\frac 12}\in \mathcal L_2(L^2).$
	The mild solution $X^N(t)$ satisfies the energy evolution law
	\begin{align*}
	\E(V_1(u^N(t),v^N(t)))
	=&V_1(u^N_0, v^N_0)
	+\frac 12{\rm{Tr}}\left(( P_N\Q^\frac 12)(P_N\Q^\frac 12)^\ast\right)t, \quad t\in \mathbb R^+.
	\end{align*}
\end{prop}

\subsection{Exponential integrability and regularity estimates of the spatial discretization}

In this part, we show the exponential integrability property of $X^N$.
In \cite[Section 5.4]{CHJ13}, the authors first obtain the exponential integrability of spectral Galerkin method of $2$-dimensional stochastic wave equation driven by multiplicative noise on a non-empty compact domain. 
For the applications of exponential integrability, we refer to \cite{BG99,BCH18,CH17,CH18,CHS18b,HK98} and references therein.

\begin{lm}
	\label{explemma;solu}
	Assume that $X_0\in \H^1,$ $T>0$ and $\Q^{\frac 12}\in \mathcal L_2(L^2).$
	Then there exist a constant $\alpha\geq \tr(\Q)$ and a positive constant $C:=C(X_0,T,\mathbf Q, \alpha)$ such that
	\begin{align}\label{exp-mom}
	\sup_{s\in[0,T]}\E\left[\exp\left(\frac{V_1(u^N(s),v^N(s))}{\exp(\alpha s)}\right)\right]\leq C.
	\end{align}
\end{lm}

\begin{proof}
	Denote 
	\begin{align*}
	(\GG_{A_N+\F_N, \mathbb G_{N}}\left(V_{1}\right))(u,v):=&
	\<D_uV_1(u,v),v\>_{L^2}+\<D_vV_1(u,v),\Lambda u-f(u)\>_{L^2}\\
	&+\frac 12\sum\limits_{i=1}^\infty \< D_{vv}V_1(u,v) P_N\Q^\frac 12e_i,P_N\Q^\frac 12 e_i\>_{L^2}.
	\end{align*}
	A direct calculation similar to $(5.43)$ in \cite[Section 5.4]{CHJ13} leads to 
	\begin{align*}
	&\left(\GG_{A_N+\F_N,\G_N}(V_{1})\right)\left(u^{N},v^{N}\right)\\
	=&\<f(u^N)-\Lambda u^N, v^N\>_{L^2}+\<v^N,P_N(\Lambda u^N-f(u^N))\>_{L^2}
	+\frac 12\tr(P_N\Q^\frac 12 (P_N\Q^\frac 12)^*)\\
	=&\frac 12\tr(P_N\Q^\frac 12 (P_N\Q^\frac 12)^*).
	\end{align*}	
	Then we get that for $\alpha>0$,
	\begin{align*}
	&\left(\GG_{A_N+\F_N,\G_N}(V_{1})\right)\left(u^{N},v^{N}\right)+\frac{1}{2\exp(\alpha t)} 
	\sum\limits_{i=1}^\infty\<(P_N\Q^\frac 12)^*v^N,e_i\>^2_{L^2}
	\\
	\leq&\frac 12\tr(\Q)+\frac{1}{2\exp(\alpha t)}\sum\limits_{i=1}^\infty\<v^N,\Q^\frac 12e_i\>^2_{L^2}
	\leq \frac 12\tr(\Q)+\frac{1}{\exp(\alpha t)}V_1(u^N,v^N)\tr(\Q).
	\end{align*}
	
	Let $\bar U=-\frac 12\tr(\Q),\alpha\geq \tr(\Q).$ According to the exponential integrability lemma in \cite[Corollary 2.4]{CHJ13}, we have
	\begin{align*}
	\E\left[\exp\left(\frac{V_1(u^N(t),v^N(t))}{\exp (\alpha t)}+\int_0^t\frac{\bar U(s)}{\exp (\alpha s)}ds \right) \right]\leq \exp (V_1(u_0^N,v_0^N)),
	\end{align*}
	which implies \eqref{exp-mom}.
\end{proof}


\begin{cor}
	\label{cor;EIsolu}
	Let  $d=1,2$. Assume that $X_0\in \H^1,$ $T>0$ and $\Q^{\frac 12}\in \mathcal L_2(L^2).$ For any $c>0$, it holds that
	\begin{align*}
	\sup_{N\in \mathbb N^+}\E\Big[\exp\left(\int_0^T c\|u^N(s)\|_{L^6}^2ds\right)\Big]<\infty.
	\end{align*}
\end{cor}
\begin{proof}
	By using Jensen's inequality, the Gagliardo--Nirenberg inequality $\|u\|_{L^6}\le C\|\nabla u\|_{L^2}^a\|u\|^{1-a}_{L^2}$ with $a=\frac d3$, and the Young inequality,
	we have that
	\begin{align*}
	&\E\left[\exp\left(\int_0^T c\|u^N(s)\|_{L^6}^2ds\right)\right)\leq \sup_{t\in[0,T]}\E\left[\exp(cT\|u^N(t)\|_{L^6}^2)\right]\\
	\leq&  \sup_{t\in [0,T]}\E\left[\exp\left(\frac{\|\nabla u^N(t)\|_{L^2}^2}{2\exp(\alpha t)}\right)\exp\Big(\exp(\frac a{1-a}\alpha T)\|u^N(t)\|_{L^2}^2 (cCT)^{\frac 1{1-a}}2^{\frac a{1-a}}\Big)\right].
	\end{align*}
	Then the H\"older and the Young inequalities imply that for sufficient small $\epsilon>0,$
	\begin{align*}
	&\E\left[\exp\left(\int_0^T c\|u^N(s)\|_{L^6}^2ds\right)\right]\\\nonumber 
	\leq&  C(\epsilon,d)\sup_{t\in [0,T]}\E\left[\exp\left(\frac{\|\nabla u^N(t)\|_{L^2}^2}{2\exp(\alpha t)}\right)\exp(\epsilon \|u^N(t)\|_{L^4}^4)\right]\\\nonumber 
	\le& C(\epsilon,d)\sup_{t\in [0,T]}\E\left[\exp\left(\frac{V_1(u^N(t),v^N(t))}{\exp(\alpha t)}\right)\right].
	\end{align*}
	Applying Lemma \ref{explemma;solu}, 
	we complete the proof.
\end{proof}
\begin{rk}
	When $d=1,$ using the Gagliardo--Nirenberg inequality  $\|u\|_{L^{\infty}}\le \|\nabla u\|_{L^2}^{\frac 12}\|u\|_{L^2}^{\frac 12}$, one can obtain that for any $c>0$,
	$
	\sup\limits_{N\in\mathbb N^+}\E\left[\exp\left(\int_0^T c\|u^N(s)\|_{L^\infty}^2ds\right)\right]<\infty.
	$
\end{rk}

Now we show the higher regularity estimate of the solution of \eqref{mod;sga}.

\begin{prop}\label{d1-reg}
	Let $p\ge1$, $d=1$, $\beta\ge 1$, $\|(-\Lambda)^{\frac {\beta-1}2}\Q^\frac 12\|_{\mathcal L_2(L^2)}<\infty$, $T>0$ and $X_0\in \H^\beta$.
	Then the mild solution of \eqref{mod;sga} satisfies
	\begin{align*}
	\|X^N\|_{L^{p}(\Omega;\mathcal C([0,T];\mathbb H^{\beta}))}\le C(X_0,T,\mathbf Q,p).
	\end{align*}
\end{prop}

\begin{proof}
	For the stochastic convolution, using the Burkholder--Davis--Gundy inequality and the unitary property of $E_N(\cdot)$, we have
	\begin{align*}
	&\E\sup\limits_{t\in[0,T]}\left\|\int_0^tE_N(t-s)\G_NdW(s)\right\|_{\mathbb H^{\beta}}^p
	\leq  C\left(\int_0^T\|(-\Lambda)^{\frac{\beta-1}{2}}\Q^{\frac 12}\|_{\LL_2(L^2)}^2ds\right)^\frac p2\leq C.
	\end{align*}
	Now it suffices to estimate $\|\int_0^tE_N(t-s)\F_N(X^N(s))ds\|_{L^{p}(\Omega;\mathcal C([0,T];\mathbb H^{\beta}))}$.
	Since
	\begin{align*}
	E_N(t-s)\F_N(X^N(s))=\begin{bmatrix}
	-(-\Lambda)^{-\frac 12}S(t-s)f(u^N(s))\\
	-C(t-s)f(u^N(s))
	\end{bmatrix},
	\end{align*}
	it suffices to estimate
	$\E\Big[\sup\limits_{t\in[0,T]}(\int_0^t\|(-\Lambda)^{\frac {\beta-1} 2}f(u^N(s))\|_{L^2}ds)^p\Big]$.
	The Sobolev embedding $\dot {\mathbb H}^1 \hookrightarrow L^\infty$ leads to
	\begin{align*}
	\int_0^t\left\|(-\Lambda)^{\frac {\beta-1} 2}f(u^N(s))\right\|_{L^2}ds
	&\le C\int_0^t(1+\|u^N(s)\|_{\dot \H^1}^2)\|u^N(s)\|_{\dot \H^{\beta -1}}ds.
	\end{align*}
	Based on the H\"older inequality and the Young inequality, we obtain
	\begin{align*}
	&\E\sup\limits_{t\in[0,T]}\left(\int_0^t\|(-\Lambda)^{\frac {\beta-1} 2}f(u^N(s))\|_{L^2}ds\right)^p\\
	\leq& C\E
	\int_0^T(1+\|u^N(s)\|_{\dot \H^1}^2)^p\|u^N(s)\|_{\dot \H^{\beta -1}}^pds\\
	\leq & C\int_0^T \E(1+\|u^N(s)\|_{\dot \H^1}^{4p})ds
	+C\int_0^T \E\|u^N(s)\|_{\dot \H^{\beta -1}}^{2p}ds,
	\end{align*}
	which, together with Lemma \ref{h1}  shows the desired result for the case that $\beta \in [1,2)$. For $\beta\in [n,n+1)$, $n\in {\{2,3,\cdots\}}$, we complete the proof via induction arguments.
\end{proof}

The following regularity estimate of $X^N$ is for the case $d=2$. 
Compared to Proposition  \ref{d1-reg},  one may  use different skill to deal with the cubic nonlinearity due to the Sobolev embedding $\dot {\mathbb H}^1 \hookrightarrow L^\infty$  fails.

\begin{prop}
	Let $d=2$, $T>0$, $X_0\in \H^2$ and $\|(-\Lambda)^{\frac {1}2}\Q^\frac 12\|_{\mathcal L_2(L^2)}<\infty.$ 
	Then for any $p\geq2$, there exists a positive constant $C:=C(X_0,\Q,T,p)$ such that
	\begin{align}
	\label{spec;H2}
	\|X^N\|_{L^p(\Omega; \mathcal C([0,T];\H^2))}\leq
	C(X_0,\Q,T,p).
	\end{align}
\end{prop}

\begin{proof}
	We only present the proof for $p=2$ here, since the proof for general $p>2$ is similar.
	Similar to the proof of Proposition \ref{d1-reg}, it only suffices to get a uniform bound of $X^N$ under the $\mathcal C([0,T];L^2(\Omega;\mathbb H^2))$-norm. 
	We introduce another Lyapunov functional 
	$
	V_2(u^N,v^N)=\frac 12\left\|\Lambda u^N\right\|^2_{L^2}
	+\frac 12\left\|\nabla v^N\right\|^2_{L^2}
	+\frac 12\<(-\Lambda)u^N,f(u^N)\>_{L^2}.
	$
	By applying It\^o's formula to $V_2$ and the commutativity between $\Lambda$ and $P_N$, we get 
	\begin{align*}
	dV_2(t)
	=&{\MyRoman{1}_1}(t)dt
	+\<\nabla v^N(t),\nabla P_NdW(t)\>_{L^2}
	+\frac 12{\rm{Tr}}\left((\nabla P_N\Q^\frac 12)(\nabla P_N\Q^\frac 12)^\ast\right)dt,
	\end{align*}
	where $${\MyRoman{1}_1}(t)=\frac 12\<\nabla u^N(t),f''(u^N(t))\nabla u^N(t)v^N(t)\>_{L^2}.$$
	Making use of the H\"older inequality and the Gagliardo--Nirenberg inequality 
	$\|\nabla u\|_{L^4}\le C\|\Lambda u\|_{L^2}^{\frac 12}\|\nabla u\|_{L^2}^{\frac 12},$ we have
	\begin{align*}
	{\MyRoman{1}_1}\leq& C\|\nabla u^N\|^2_{L^4}(1+\|u^N\|_{L^{\infty}})\|v^N\|_{L^2}
	\leq C\|\Lambda u^N\|_{L^2}\|\nabla u^N\|_{L^2}(1+\|u^N\|_{L^{\infty}})\|v^N\|_{L^2}.
	\end{align*}
	By further applying the Gagliardo--Nirenberg inequality
	$\|u\|_{L^{\infty}}\le C\|\Lambda u\|^{\frac 14}_{L^2}\|u\|_{L^6}^{\frac 34}$ 
	and using the Young inequality, we get
	\begin{align*}
	{\MyRoman{1}_1}
	&\le C\|\Lambda u^N\|_{L^2}\|\nabla  u^N\|_{L^2}(1+\|\Lambda u^N\|_{L^2}^{\frac 14}\|u^N\|_{L^6}^{\frac 34})\|v^N\|_{L^2}\\
	&\le C\Big(\|\nabla u^N\|_{L^2}^2\|v^N\|_{L^2}^2+\|\nabla u^N\|_{L^2}^{\frac 83}\|v^N\|_{L^2}^{\frac 83}\|u^N\|_{L^6}^{2}+\|\Lambda u^N\|_{L^2}^{2}\Big).
	\end{align*}
	Using the Cauchy--Schwarz inequality and the Young inequality and the fact that $\dot{\H}^1\hookrightarrow L^6$, we deduce that 
	\begin{align*}
	&\left|\<(-\Lambda)u^N,f(u^N)\>_{L^2}\right|\\
	\leq&
	\|(-\Lambda)u^N\|_{L^2}\|f(u^N)\|_{L^2}\leq
	\frac 12\|\Lambda u^N\|_{L^2}^2+\frac {\tilde C_1(c_0,c_1,c_2,c_3)}{2}\|u^N\|_{L^6}^6+{\tilde C_2(c_0,c_1,c_2,c_3)}\\
	\leq& \frac 12\|\Lambda u^N\|_{L^2}^2+\frac {\tilde C_1(c_0,c_1,c_2,c_3,d)}{2}\|u^N\|_{\dot{\H}^1}^6+\tilde C_2(c_0,c_1,c_2,c_3).
	\end{align*}
	The above inequality leads to $V_2(u^N,v^N)\geq \frac 14 \|\Lambda u^N\|_{L^2}^2-\frac {\tilde C_1}{4}\|u^N\|_{\dot{\H}^1}^6-\frac 12\tilde C_2,$ which yields that 
	\begin{align*}
	dV_2\leq&
	C\left(V_2+\|u^N\|_{\dot{\H}^1}^6+1\right)dt
	+C(1+\|u^N\|_{\dot{\H}^1}^8+\|v^N\|_{L^2}^8)dt\\
	&+\<\nabla v^N,\nabla P_NdW\>_{L^2}
	+\frac 12{\rm{Tr}}\left((\nabla P_N\Q^\frac 12)(\nabla P_N\Q^\frac 12)^\ast\right)dt.
	\end{align*}
	By using the inverse equality $\|u^N\|_{\dot{\H}^{2}}\le C\lambda_N \|u^N\|_{L^2}$ and the integrability of $u^N$ in $\mathbb H^1$ in  \eqref{spec;H1}, one could obtain the integrability of $V_2$ and other terms on the right hand side of the above equality. 
	Taking the expectation on both sides and applying the Gronwall inequality in \cite[Corollary 2]{Drag03}, we have
	\begin{align*}
	\E V_2(u^N(t),v^N(t))\leq& C\exp{(Ct)}
	\Big(\|X_0\|_{\mathbb H^2}^2+\frac 12{\rm{Tr}}\left((\nabla P_N\Q^\frac 12)(\nabla P_N\Q^\frac 12)^\ast\right)t\\
	&+\int_0^t\E(1+\|v^N(s)\|_{L^2}^8+\|u^N(s)\|_{\dot{\mathbb H}^1}^8)ds\Big),
	\end{align*}
	which, combined with Lemma \ref{h1}, shows the desired result. 
\end{proof}

Next we derive the H\"{o}lder continuity in temporal direction for the numerical solution $\{u^N\}_{N\in\N}$ and $\{X^N\}_{N\in\N}$ with respect to $L^p(\OOO;L^2)$-norm and $L^p(\OOO;\H)$-norm, respectively. 
Both results play a key role in our error analysis in Section \ref{sec4}.

\begin{lm}
	\label{lm;Holderu^N}
	Assume that conditions in  Lemma \ref{h1} hold. Then there exists $C:=C(X_0,\mathbf Q, T,p)>0$ such that for any $0\leq s\leq t\leq T$,
	\begin{equation*}
	\begin{split}
	\|u^N(t)-u^N(s)\|_{L^p(\OOO;L^2)}\leq C|t-s|,\quad \|X^N(t)-X^N(s)\|_{L^p(\OOO;\H)}\leq C|t-s|^{\frac 12}.
	\end{split}
	\end{equation*}
\end{lm}

\begin{proof}
	From \eqref{mod;sga}, we have
	\begin{align*}
	u^N(t)-u^N(s)=&(C_N(t)-C_N(s))P_N(u_0)+(-\Lambda_N)^{-\frac 12}(S_N(t)-S_N(s))P_N(v_0)\\
	&-\int_0^s(-\Lambda_N)^{-\frac 12}(S_N(t-r)-S_N(s-r))P_N(f(u^N))dr\\
	&-\int_s^t(-\Lambda_N)^{-\frac 12}S_N(t-r)P_N(f(u^N))dr\\
	&+\int_0^s(-\Lambda_N)^{-\frac 12}(S_N(t-r)-S_N(s-r))P_NdW(r)\\
	&+\int_s^t(-\Lambda_N)^{-\frac 12}S_N(t-r)P_NdW(r).
	\end{align*}
	Using the properties of $C_N(t)$ and $S_N(t)$ in Lemma \ref{lm;prpE_n} and the Burkholder--Davis--Gundy inequality gives 
	\begin{align*}
	&\|u^N(t)-u^N(s)\|_{L^p(\OOO;L^2)}\\
	\leq& C|t-s|\left(\|u_0\|_{L^p(\OOO;\dot{\H}^1)}+\|v_0\|_{L^p(\OOO;L^2)}\right)\\
	&+C\int_0^s(t-s)\|f(u^N)\|_{L^p(\OOO;L^2)}ds
	+C\int_s^t\|f(u^N)\|_{L^p(\OOO;\dot{\H}^{-1})}ds\\
	&+\left(\int_0^s\|(-\Lambda_N)^{-\frac 12}(S_N(t-r)-S_N(s-r))P_N\mathbf Q^\frac 12\|_{\LL_2(L^2)}^2dr\right)^{\frac 12}\\
	&+\left(\int_s^t\|(-\Lambda_N)^{-\frac 12}S_N(t-r)P_N\mathbf Q^\frac12\|_{\LL_2(L^2)}^2dr\right)^{\frac 12}\\
	\leq& C|t-s|\left(1+\|u_0\|_{L^p(\OOO;\dot{\H}^1)}+\|v_0\|_{L^p(\OOO;L^2)}+\sup_{0\leq t\leq T}\|u^N(t)\|^3_{L^{3p}(\OOO;\dot{\H}^1)}\right)\le C|t-s|,
	\end{align*}
	which is the claim for $u^N$.
	For the estimate of $X^N$, the proof is similar. 
\end{proof}

\subsection{Exponential integrability of stochastic wave equation}
In this part, we first prove that the spectral Galerkin approximation $X_N$ converges to the solution of \eqref{mod;swe1} in strong sense 
based on Lemma \ref{explemma;solu}.

\begin{prop}\label{con-spa}
	Assume that $d=1$, $\beta\ge 1$ {\rm (}or $d=2$, $\beta=2${\rm )}, $X_0\in \H^\beta,$ $T>0$ and $\|(-\Lambda)^{\frac {\beta-1}2}\Q^\frac 12\|_{\mathcal L_2(L^2)}$$<\infty$.
	Then for any $p\geq 2,$ \eqref{mod;sga} satisfies that 
	\begin{align*}
	&\|X^N-X\|
	_{ L^p(\OOO;\mathcal C([0,T];\mathbb H))}
	=O(\lambda_N^{-\frac \beta 2}).
	\end{align*}
\end{prop}

\begin{proof}
	For the sake of simplicity, we take the first component of $X^N$ as an example to illustrate the desired result.
	
	{\em Step 1: Strong convergence and limit of $u^N$.}
	We claim that $\{u^N\}_{N\in\N^+}$ is a Cauchy sequence in $ L^p(\OOO;\mathcal C([0,T];L^2)).$
	Notice that
	\begin{align*}
	u^N(t)-u^{M'}(t)&=\left(u^N(t)-P_Nu^{M'}(t)\right)+\left((P_N-I)u^{M'}(t)\right),
	\end{align*}
	where $N,{M'}\in \N^+.$
	Without loss of generality, it may be assumed that ${M'}>N.$
	According to the expression of both $u^N$ and $u^{M'},$ using the definition of $P_N,$ we have
	\begin{align*}
	\|(P_N-I)u^{M'}(t)\|_{L^2}^2=\sum\limits_{i=N+1}^{\infty}
	\lambda_i^{-\beta}\<u^{M'}(t),\lambda_i^{\frac\beta 2}e_i\>^2_{L^2}
	\leq \lambda_N^{-\beta}\|u^{M'}(t)\|_{{\dot \H}^\beta}^2
	\end{align*}
	with $\beta\geq 1.$
	With respect to the term $u^N(t)-P_Nu^{M'}(t),$ we have
	\begin{align*}
	u^N(t)-P_Nu^{M'}(t)
	=&\int_0^t(-\Lambda)^{(-\frac 12)}S(t-s)P_N\left(f(u^{M'})-f(u^N)\right)ds.
	\end{align*}
	From the Sobolev embedding $L^\frac 65\hookrightarrow \dot{\H}^{-1}$, and using the H\"older inequality,
	\begin{align*}
	&\|u^N(t)-P_Nu^{M'}(t)\|_{L^2}\leq \int_0^t\left\|(-\Lambda)^{(-\frac 12)}S(t-s)P_N\left(f(u^{M'})-f(u^N)\right) \right\|_{L^2}ds\\
	\leq &C\int_{0}^t
	(1+\|u^N\|_{L^6}^2+\|u^{M'}\|_{L^6}^2)\Big(\|u^N-P_Nu^{M'}\|_{L^2}+\|(P_N-I)u^{M'}\|_{L^2}\Big)ds,
	\end{align*}
	which implies
	\begin{align*}
	&\|u^N(t)-P_Nu^{M'}(t)\|_{L^2}\\
	\leq &
	C\lambda_N^{-\frac\beta 2}\exp\left(\int_0^T \left(\|u^N\|_{L^6}^2+
	\|u^{M'}\|_{L^6}^2\right)ds\right)\int_{0}^t
	(1+\|u^N\|_{L^6}^2+
	\|u^{M'}\|_{L^6}^2)\|u^{M'}\|_{{\dot \H}^\beta}ds
	\end{align*}
	due to Gronwall's inequality.
	Taking the $p$th moment and then using the  H\"older and the Young inequalities, and Corollary \ref{cor;EIsolu}, we obtain
	\begin{align*}
	&\|u^N-u^{M'}\|_{ L^p(\OOO;\mathcal C([0,T];L^2))}
	\\
	\leq&
	C\lambda_N^{-\frac\beta 2}\Big\|\exp\left(\int_0^T \left(\|u^N\|_{L^6}^2+
	\|u^{M'}\|_{L^6}^2\right)ds\right)\Big\|_{L^{2p}{(\Omega;\R)}}\\
	&\cdot
	\Big\|\int_{0}^T
	(1+\|u^N\|_{L^6}^2+
	\|u^{M'}\|_{L^6}^2)\|u^{M'}\|_{{\dot \H}^\beta}ds\Big\|_{L^{2p}{(\Omega;\R)}}+\lambda_N^{-\frac{\beta}2}\|u^{M'}\|_{L^p(\Omega; \mathcal C([0,T];{\dot \H}^\beta))},
	\end{align*}
	which leads to
	\begin{align*}
	\|u^N-u^{M'}\|_{ L^p(\OOO;\mathcal C([0,T];L^2))}
	\leq C\lambda_{N}^{-\frac\beta 2}.
	\end{align*}
	Similarly, we can prove that $\{v^{M'}\}_{{M'}\in \N^+}$ is a Cauchy sequence in $ L^p(\OOO;\mathcal C([0,T];\dot\H^{-1}))$ which means that $\{X^{M'}\}_{{M'}\in \N^+}$ is a Cauchy sequence in $ L^p(\OOO;\mathcal C([0,T];\H)).$ 
	Denote by $X=(u,v)^{\top}\in \H$ the limit of $\{X^{M'}\}_{{M'}\in \N^+}$.
	Due to Propositions \ref{d1-reg} and Fatou's lemma, we  have  $\E[\|X\|_{\mathcal C([0,T];\H^1)}^p]\le C(X_0,\Q,T,p)$. From the Gagliardo--Nirenberg inequality and the boundedness of $X$ and $X^N$ in $ L^p(\OOO;\mathcal C([0,T];\H^{1})),$ it follows that $u^N$ converges to $u$ in $ L^p(\OOO;\mathcal C([0,T];L^6)).$
	By Jensen's inequality and Fatou's lemma, we have
	\begin{align*}
	\E\left(\exp\left(\int_0^Tc\|u(s)\|_{L^6}^2ds\right)\right)
	&\leq \frac 1T\int_{0}^T
	\E\exp\left(cT\|u(s)\|_{L^6}^2\right)ds\\
	&\leq\varliminf _{N\rightarrow\infty} \frac 1T\int_{0}^T
	\E\exp\left(cT\|u^N(s)\|_{L^6}^2\right)ds.
	\end{align*}
	Then the similar procedure on the proof of Corollary \ref{cor;EIsolu} leads that for any $c>0$,
	\begin{align*}
	\E\left(\exp\left(\int_0^Tc \|u(s)\|_{L^6}^2ds\right)\right)
	&<\infty.
	\end{align*}
	
	{\em Step 2: Existence and uniqueness of the mild solution.}
	To show that the strong limit $X$ is the mild solution of \eqref{mod;swe}, it suffices to prove that
	\begin{align}
	\label{sol;swe1}
	X(t)=E(t)X_0+\int_0^tE(t-s)\F(X(s))ds+
	\int_0^tE(t-s)\G dW(s)
	\end{align}
	for any $t\in[0,T].$
	We take the convergence of $\{u^N\}_{N\in\N^+}$ as an example for convenience, that is, to show that $u$ satisfies
	\begin{align*}
	u(t)=&C(t)u_0+(-\Lambda)^{(-\frac 12)}S(t)v_0-
	\int_0^t(-\Lambda)^{(-\frac 12)}S(t-s)f(u(s))ds\\
	&+\int_0^t(-\Lambda)^{(-\frac 12)}S(t-s)dW(s).
	\end{align*}
	To this end, we show that the mild form of the exact solution $u^N$ is convergent to that of $u$.
	The assumption on $X_0$ yields that
	\begin{align*}
	\|C(t)(I-P_N)u_0\|_{L^2}+\|(-\Lambda)^{(-\frac 12)}S(t)(I-P_N)v_0\|_{L^2}
	&\le C\lambda_N^{-\frac \beta 2}(\|u_0\|_{\dot \H^{\beta}}+\|v_0\|_{\dot \H^{\beta-1}}).
	\end{align*}
	Based on the Sobolev embedding $L^{\frac 65}\hookrightarrow\dot{\H}^{-1},$
	we have
	\begin{align*}
	&\left\|\int_0^t(-\Lambda)^{(-\frac 12)}S(t-s)P_N\left(f(u)-f(u^N)\right)ds\right\|_{L^p(\Omega;\mathcal C([0,T];L^2))}\\
	\leq&
	C\int_{0}^T
	\|(1+\|u\|_{L^6}^2+\|u^N\|_{L^6}^2)\|_{L^{2p}(\Omega;\R)}\|u-u^N\|_{L^{2p}(\Omega;L^2)}ds\le C\lambda_{N}^{-\frac \beta 2}.
	\end{align*}
	For the stochastic term, by the Burkholder--Davis--Gundy inequality we obtain that for $p\ge 2$,
	\begin{align*}
	&\|\int_0^t(-\Lambda)^{(-\frac 12)}S(t-s)(I-P_N)dW(s)\|_{L^p(\Omega;\mathcal C([0,T];L^2))}\\
	\le& C\|\int_0^t(-\Lambda)^{(-\frac 12)}C(s)(I-P_N)dW(s)\|_{L^p(\Omega;\mathcal C([0,T];L^2))}\\
	&+C\|\int_0^t(-\Lambda)^{(-\frac 12)}S(s)(I-P_N)dW(s)\|_{L^p(\Omega;\mathcal C([0,T];L^2))}
	\le C\lambda_N^{-\frac \beta 2}.
	\end{align*}
	Combining the above estimates, we complete the proof.
\end{proof}

From the proof of Proposition \ref{con-spa}, we have the following exponential integrability property of the exact solution.
\begin{prop}
	Let  $d=1,2,$ $X_0\in \H^1$ and $\Q^{\frac 12}\in \mathcal L_2(L^2).$ For any $c\in\mathbb R$ and $T>0$, it holds that $
	\E\left(\exp\left(\int_0^T c\|u(s)\|_{L^6}^2ds\right)\right)<\infty.
	$
\end{prop}

\section{An energy-preserving exponentially integrable full discrete method}
\label{sec4}

The stochastic wave equation with Lipschitz and regular coefficients has been systematic investigated theoretically and numerically (see e.g., \cite{ACLW16,DS09,KLS10,PZ00} and references therein).
However, as far as we know, for stochastic wave equations with non-globally Lipschitz coefficients, there are no results about the full discretization preserving both the energy evolution law and the exponential integrability property and the related strong convergence analysis.
In this section, we propose an energy-preserving exponentially integrable  full discretization for stochastic wave equation \eqref{mod;swe} by applying splitting AVF method to \eqref{mod;sga}, and finally obtain a strong  convergence theorem for the full discrete numerical method.
Let $N,M \in \mathbb N^{+}$ and $T=Mh$ and denote $\Z_{M+1}=\{0,1,\dots,M\}$. 
For any $T>0$, we partition the time domain $[0,T]$ uniformly with nodes $t_m=mh,$ $m=0,1,\cdots M$ for simplicity. One may use the non-uniform discretization, and the analysis of both the convergence and structure-preserving properties is similar.

We first decompose \eqref{mod;swe} into a deterministic system on $[t_m,t_{m+1}],$
\begin{align}
\label{mol;dswe1}
&du^{N,D}(t) = v^{N,D} (t) dt,\quad
dv^{N,D} (t) =\Lambda_N u^{N,D}(t)  dt-P_N(f(u^{N,D}(t) ))dt,
\end{align}
and a stochastic system on $[t_m,t_{m+1}],$
\begin{align}
\label{mol;dswe2}
&du^{N,S}(t) = 0,\quad
dv^{N,S} (t)=P_NdW(t),\\
&u^{N,S} (t_m) =u^{N,D} (t_{m+1}),\quad
v^{N,S} (t_m)=v^{N,D} (t_{m+1}).\nonumber
\end{align}
Then on each subinterval  $[t_m,t_{m+1}]$, 
$u^{N,S}(t)$ starting from $u^{N,S}(t_m)=u^{N,D}(t_{m+1})$ and $v^{N,S}(t)$ starting from $v^{N,S}(t_m)=v^{N,D}(t_{m+1})$ 
can be informally viewed as approximations of $u^N(t)$ with $u^N(t_m)=u^{N,D}(t_m)$ and $v^N(t)$ with $v^N(t_m)=v^{N,D}(t_{m})$ in \eqref{mod;swe}, respectively.
By further using  the explicit solution of \eqref{mol;dswe2} and the AVF method to discretize \eqref{mol;dswe1}, we obtain the splitting AVF method
\begin{equation}
\begin{split}
\label{sche;sbes}
u^N_{m+1}=&u^N_{m}+h\bar v^N_{m+\frac 12},\\
\bar v^N_{m+1}=&v^N_m+h\Lambda_N u^N_{m+\frac 12}-hP_N\left(\int_0^1f(u^N_m+\theta(u^N_{m+1}-u^N_m))d\theta \right),\\
v^N_{m+1}=&\bar v^N_{m+1}+P_N\delta W_{m},
\end{split}
\end{equation}
where $u_0^N=P_Nu_0,v_0^N=P_Nv_0$, $\bar v^N_{m+\frac 12}=\frac 12(\bar v^N_{m+1}+v^N_{m})$, $u^N_{m+\frac 12}=\frac 12(u^N_{m+1}+u^N_{m})$ and the increment $\delta W_{m}:=W(t_{m+1})-W(t_{m})=\sum_{k=1}^{\infty}(\beta_k(t_{m+1})-\beta_k(t_m))\Q^{\frac{1}{2}}e_k.$

Denote 
$$\AA(t)=\begin{pmatrix}
I & \frac t2 I\\
\Lambda_N\frac t2 & I
\end{pmatrix},\quad
\B(t)=\begin{pmatrix}
I & -\frac t2 I\\
-\Lambda_N\frac t2 & I
\end{pmatrix}$$
and $\MM(t)=I-\Lambda_N\frac {t^2}4$. Then we have
\begin{align*}
\B^{-1}(t)\AA(t)=
\begin{pmatrix}
\MM^{-1}(t) & 0\\
0 & \MM^{-1}(t)
\end{pmatrix}
\AA^2(t)
=\begin{pmatrix}
2\MM^{-1}(t)-I & \MM^{-1}(t)t\\
\MM^{-1}(t)\Lambda_Nt & 2\MM^{-1}(t)-I
\end{pmatrix}.
\end{align*}
This formula yields  that \eqref{sche;sbes} can be rewritten as
\begin{align}
\label{sch;mild2}
\begin{pmatrix}
u^N_{m+1}\\
v^N_{m+1}
\end{pmatrix}
&=\B^{-1}(h)\AA(h)
\begin{pmatrix}
u^N_{m}\\
v^N_{m}
\end{pmatrix}\\
&+\B^{-1}(h)
\begin{pmatrix}
0\\
-hP_N\left(\int_0^1f(u^N_m+\theta(u^N_{m+1}-u^N_m))d\theta\right)
\end{pmatrix}
+\begin{pmatrix}
0\\
P_N(\delta W_{m})
\end{pmatrix}\nonumber.
\end{align}

For convenience, we assume that there exists a sufficiently small $h_0>0$ which is not depending on $N$ such that the numerical solution of \eqref{sche;sbes} exists and is unique  (see more details in Appendix). Throughout this section, we always require that the temporal step size $h\le h_0.$ 
When performing the numerical scheme \eqref{sche;sbes}, some iterations procedures are used to approximate \eqref{sche;sbes} since the AVF scheme is implicit. 
To study the strong convergence of the proposed numerical method, we first give some estimates of the matrix $\B^{-1}(\cdot)\AA(\cdot).$

\begin{lm}
	\label{lm;BINAest}
	For any $r\geq 0,$ $t\geq 0,$ and $w\in \H^r,$ 
	$
	\|\B^{-1}(t)\AA(t)w\|_{\H^r}=\|w\|_{\H^r}.
	$
\end{lm}

Following  \cite[Theorem 3]{BT79}, we give the following
lemma which is applied to the error estimate for \eqref{sche;sbes}.

\begin{lm}
	\label{err;era}
	For any $r\geq 0$ and $h\geq 0$, there exists a positive constant $C:=C(r)$ such that for any $w\in \H^{r+2},$
	\begin{equation}
	\begin{split}
	\label{est;E_N_2}
	&\|(E_N(h)-\B^{-1}(h)\AA(h))w\|_{\H^r}\leq  Ch^2\|w\|_{\H^{r+2}},\\[1mm]
	&\|(E_N(h)-\B^{-1}(h))w\|_{\H^r}\leq Ch\|w\|_{\H^{r+1}}.
	\end{split}
	\end{equation}
\end{lm}

\begin{prop}
	\label{prop;REGSP1}
	Assume that $T>0$, $p\ge 1$, $X_0\in \H^1$ and $\Q^{\frac 12}\in \mathcal L_2(L^2).$
	Then the solution of \eqref{sche;sbes} satisfies 
	\begin{align}
	\label{15}
	\sup_{m\in \Z_{M+1}}\E(V_1^p(u^N_{m}, v^N_{m}))\leq C,
	\end{align}
	where $N\in \N^+$, $Mh=T,$ $M\in \mathbb N^+,$  $C=C(X_0,\Q,T,p)>0$ and $V_1(u^N_{m}, v^N_{m}):=\frac 12\|u^N_m\|_{\dot{\mathbb{H}}^1}^2+\frac 12 \|v^N_m\|_{L^2}^2+F(u^N_m)+C_1$ with $C_1>b_1.$
\end{prop}
\begin{proof}
	Fix $t\in T_m:=[t_{m},t_{m+1}]$ with $m\in\Z_{M}.$   Since on the interval $T_m,$ $u^{N,S}(t)=u^N_{m+1}, v^{N,S}(t_m)=\bar v^N_{m+1}$, we have 
	\begin{align*}
	V_1(u^N_{m+1},v^N_{m+1})=&
	V_1(u^N_{m+1},\bar v^{N}_{m+1})
	+
	\int_{t_m}^{t_{m+1}}\<v^{N,S} (s), P_NdW(s)\>_{L^2}\\
	&+\int_{t_m}^{t_{m+1}}\frac 12{\rm{Tr}}\left(( P_N\Q^\frac 12)( P_N\Q^\frac 12)^\ast\right)ds.
	\end{align*}
	Due to the fact that $d u^{N,S}(t)=0$ and $dv^{N,S}(t)=P_NdW(t),$
	we can apply It\^o's formula to $V_1^p(u^{N,S} (t),v^{N,S} (t))$ and obtain that for $p\ge 2$
	\begin{align*}
	V_1^p&(u^{N,S} (t),v^{N,S} (t))
	=V_1^p(u^{N,S} (t_m),v^{N,S} (t_m))\\
	&+\frac p2\int_{t_m}^t
	V_1^{p-1}(u^{N,S} (s),v^{N,S} (s)){\rm{Tr}}\left(( P_N\Q^\frac 12)( P_N\Q^\frac 12)^\ast\right)ds\\
	&+p\int_{t_m}^t
	V_1^{p-1}(u^{N,S} (s),v^{N,S} (s))\< v^{N,S}(s) , P_NdW(s)\>_{L^2}\\
	&+\frac {p(p-1)}2\sum\limits_{i=1}^N\int_{t_m}^t
	V_1^{p-2}(u^{N,S} (s),v^{N,S} (s))\<v^{N,S}(s), \Q^{\frac 12}e_i\>_{L^2}^2ds.
	\end{align*}
	Taking the expectation on both sides of the above equation, using the martingality of the stochastic integral and applying the H\"older and Young inequalities,
	\begin{align*}
	\E(V_1^p(u^{N,S} (t),v^{N,S} (t)))
	\leq&\E(V_1^p(u^N_{m+1}, \bar v^N_{m+1}))\\
	&+C\int_{t_m}^t(1+\E(V_1^p(u^{N,S} (s),v^{N,S} (s))))ds,
	\end{align*}
	which, together with the Gronwall inequality in \cite[Corollary 3]{Drag03} and the property 
	\begin{align}\label{sdcn}
	V_1(u^N_{m+1},\bar v^N_{m+1})=V_1(u^N_m,v^N_m),
	\end{align} leads to
	$\E(V_1^p(u^N_{m+1}, v^N_{m+1}))\leq \exp(Ch)\left(\E(V_1^p(u^{N}_m,v^{N}_m))+Ch\right).$ 
	Since $Mh=T$, iteration arguments lead to
	\begin{align*}
	\sup_{m\in \Z_{M}}\E(V_1^p(u^N_{m+1}, v^N_{m+1}))\leq \exp(CT)\E(V_1^p(u^{N}_0,v^{N}_0))+\exp(CT)CT,
	\end{align*}
	which implies the estimate \eqref{15}.
\end{proof}

From the above proof of Proposition \ref{prop;REGSP1}, we get the following theorem which shows that the proposed method preserves the evolution law of the energy $V_1$ in \eqref{Lya-V1}.

\begin{theorem}
	Assume that $T>0$, $X_0\in \H^1$ and $\Q^{\frac 12}\in \mathcal L_2(L^2).$
	Then the solution of \eqref{sche;sbes} satisfies 
	\begin{align*}
	\E(V_1(u^N_m,v^N_m))
	=&V_1(u^N_0, v^N_0)
	+\frac 12{\rm{Tr}}\left(( P_N\Q^\frac 12)(P_N\Q^\frac 12)^\ast\right)t_m,
	\end{align*}
	where $N\in \N^+, m\in \mathbb Z_{M+1}$, $Mh=T,$ $M\in \mathbb N^+,$ and $V_1(u^N_{m}, v^N_{m}):=\frac 12\|u^N_m\|_{\dot{\mathbb{H}}^1}^2+\frac 12 \|v^N_m\|_{L^2}^2+F(u^N_m)+C_1$ with $C_1>b_1.$
\end{theorem}

Beside the energy-preserving property, the proposed numerical method also inherits the exponential integrability property of the original system as following.

\begin{prop}
	\label{lm;expoSp}
	Let $d=1,2,$ $X_0\in\H^1$, $T>0$ and $\|\Q^\frac 12\|_{\mathcal L_2(L^2)}<\infty.$ Then the solution of \eqref{sche;sbes} satisfies 
	\begin{align}\label{16}
	\E
	\left[\exp\left(ch\sum\limits_{i=0}^m\|u^N_i\|_{L^6}^2
	\right)
	\right]
	\leq C
	\end{align}
	for any $c>0$, where 
	$C:= C(X_0,\Q,T, c)>0$, $N\in \N^+, m\in \mathbb Z_{M+1}$, $M\in \mathbb N^+, Mh=T.$
\end{prop}

\begin{proof}
	Notice that $$V_1(u^N_{m+1},v^N_{m+1})=V_1(u^{N,S}(t_{m+1}),v^{N,S}(t_{m+1})),$$
	where $v^{N,S}(t_{m+1})$ is the solution of \eqref{mol;dswe2} defined on $[t_m,t_{m+1}]$ with $v^{N,S}(t_{m})=\bar v^N_{m+1}$  and $u^{N,S}(t_{m})=u^N_{m+1}.$ 
	Let $\widetilde \F_N:=(0,0)^\top,$ and $\widetilde \G_N:=(0,P_N)^\top.$ 
	Then for $\alpha>0,$ 
	\begin{align*}
	&\left(\GG_{\widetilde \F_N, \widetilde \G_N}(V_1)\right)(u^{N,S},v^{N,S})+\frac{1}{2\exp(\alpha t)}  \sum\limits_{i=1}^\infty\<(P_N\Q^\frac 12)^*v^{N,S},e_i\>^2_{L^2}\\
	\leq& \frac 12\tr(\Q)+\frac{1}{\exp(\alpha t)}V_1(u^{N,S},v^{N,S})\tr(\Q).
	\end{align*}
	Let $\bar U=-\frac {1}2\tr(\Q),$ $\alpha\geq \tr(\Q).$ 
	Applying It\^o's formula 
	and taking conditional expectation, we have
	\begin{align*}
	&\quad\E\left[\exp\left(\frac{ V_1(u^{N,S}(t),v^{N,S}(t))}{\exp (\alpha t)}+\int_{t_m}^t\frac{\bar U(s)}{\exp (\alpha s)}ds \right) \right]\\
	&=\E\left[\E\left[\exp\left(\frac{ V_1(u^{N,S}(t),v^{N,S}(t))}{\exp (\alpha t)}+\int_{t_m}^t\frac{\bar U(s)}{\exp (\alpha s)}ds \right) \Big| \mathcal F_{t_m}\right]\right]\\
	&\leq \E \Big[\exp \Big(\frac {V_1(u_{m+1}^N,\bar v_{m+1}^N)}{\exp (\alpha t_m)}\Big)\Big]=\E \Big[\exp \Big(\frac {V_1(u^N_{m}, v^N_{m})}{\exp (\alpha t_m)}\Big)\Big],
	\end{align*}
	where we use the fact that on $[t_m,t_{m+1}]$, $v^{N,S}(t_{m})=\bar v^N_{m+1}$  and $u^{N,S}(t_{m})=u^N_{m+1}$ and that the energy preservation of the AVF method, $V_1(u^N_{m+1},\bar v^N_{m+1})=V_1(u^N_{m}, v^N_{m})$.
	
	Repeating the above arguments on every subinterval $[t_{l},t_{l+1}]$, $l\le m-1$, we obtain 
	\begin{align}\label{med-ex}
	\E\left[\exp\left(\frac{V_1(u_{m+1}^{N},v_{m+1}^{N})}{\exp (\alpha t_{m+1})}\right)\right]\le\exp\left( V_1(u_0^N,v^{N}_0)\right) \exp\left(-\int_{0}^{t_{m+1}}\frac{\bar U(s)}{\exp (\alpha s)}ds \right). 
	\end{align}
	Now, we are in a position to show \eqref{16}.
	By using Jensen's inequality, the Gagliardo--Nirenberg inequality $\|u\|_{L^6}\le C\|\nabla u\|_{L^2}^{a}\|u\|^{1-a}_{L^2}$ with  $a=\frac d3$, and the Young inequality,
	we have that
	\begin{align*}
	&\E\left[\exp\left(ch\sum\limits_{i=0}^m\|u^N_i\|_{L^6}^2\right)\right]\leq \sup_{i \in \Z_{M+1}}\E\left[\exp(cT\|u^N_i\|_{L^6}^2)\right]\\
	\leq&  \sup_{i \in \Z_{M+1}}\E\left[\exp\left(\frac{\|\nabla u^N_i\|_{L^2}^2}{2\exp(\alpha t_i)}\right)\exp\Big(\exp(\frac a{1-a}\alpha T)\|u^N_i\|_{L^2}^2 (cCT)^{\frac 1{1-a}}2^{\frac a{1-a}}\Big)\right].
	\end{align*}
	Then the H\"older and the Young inequalities imply that for small enough $\epsilon>0$,
	\begin{align*}
	\E\left[\exp\left(ch\sum\limits_{i=0}^m\|u^N_i\|_{L^6}^2\right)\right]
	&\leq  C(\epsilon,d)\sup_{i \in \Z_{M+1}}\E\left[\exp\left(\frac{\|\nabla u^N_i\|_{L^2}^2}{2\exp(\alpha t_i)}\right)\exp(\epsilon \|u^N_i\|_{L^4}^4)\right]\\\nonumber 
	&\le C(\epsilon,d) \sup_{i \in \Z_{M+1}}\E\left[\exp\left(\frac{V_1(u^N_i,v^N_i)}{\exp(\alpha t_i)}\right)\right].
	\end{align*}
	By applying \eqref{med-ex}, we complete the proof.
\end{proof}

Based on the above exponential integrability property of  $\{u^N_i\}_{1\leq i \leq M}$ and Lemma \ref{lm;Holderu^N}, we obtain the strong convergence rate in temporal direction as following.

\begin{rk}
	Let $d=2$.
	Assume that $X_0\in \H^2,$ $T>0$ and $\|(-\Lambda)^{\frac {1}2}\Q^\frac 12\|<\infty.$ 
	By introducing the Lyapunov functional 
	\begin{align*}
	V_2(u^N_m,v^N_m)=\frac 12\left\|\Lambda u^N_m\right\|^2_{L^2}
	+\frac 12\left\|\nabla v^N_m\right\|^2_{L^2}
	+\frac 12\<(-\Lambda)u^N_m,f(u^N_m)\>_{L^2},
	\end{align*}
	similar arguments in the proof of \cite[Lemma 3.3]{CHLZ17b} yield that for any $p\geq1$, there exists a constant $C=C(X_0,\mathbf Q,T,p)>0$ such that
	$
	\mathbb E\left[\sup\limits_{m\in\mathbb Z_{M+1}}\|u^N_m\|_{\dot{\mathbb H}^2}^p\right]\leq C.
	$
	We leave the details to readers. 
\end{rk}

\begin{prop}
	\label{tm;2}
	Let $d=1$, $\beta\ge 1$ {\rm (}or $d=2$, $\beta=2${\rm )}, $\gamma=\min{(\beta,2)}$ and $T>0.$ 
	Assume that $X_0\in \H^\beta,$ $\|(-\Lambda)^{\frac {\beta-1}2}\Q^\frac 12\|_{\mathcal L_2(L^2)}<\infty$. There exists $h_0>0$ such that for $h\le h_0$,  $p\geq1$, 
	\begin{align}
	\sup\limits_{m\in \Z_{M+1}}\E\left[\|X^N(t_m)-X^N_m\|^{2p}\right]\leq Ch^{\gamma p},
	\end{align}
	where $C:=C(p,X_0,\Q,T)>0$, $N\in \N^+$, $M\in \mathbb N^+, Mh=T.$
\end{prop}
\begin{proof}
	Let $\varepsilon_i=X^N(t_{i})-X^N_{i}$ for $i\in\mathbb Z_{M+1}$. 
	Fix $m\in\{1,\cdots,M\}.$ 
	From the equations of $X^N(t_{m})$ and $X^N_{m}$ it follows that 
	\begin{align*}
	\label{err}
	\varepsilon_{m}=
	&(E(t_m)-(\B^{-1}(h)\AA(h))^m)X^N_0\\
	&+\sum\limits_{j=0}^{m-1}
	\int_{t_j}^{t_{j+1}}
	\left(
	E(t_{m}-s)
	-(\B^{-1}(h)\AA(h))^{(m-1-j)}
	\right)
	\begin{pmatrix}
	0\\
	P_NdW(s)
	\end{pmatrix}
	\\
	&+
	\sum\limits_{j=0}^{m-1}
	\int_{t_j}^{t_{j+1}}
	\Bigg(E(t_{m}-s)
	\begin{pmatrix}
	0\\
	P_N(f(u^N(s)))
	\end{pmatrix}
	-(\B^{-1}(h)\AA(h))^{(m-1-j)}\\
	&\quad\cdot\B^{-1}(h)
	\begin{pmatrix}
	0\\
	P_N(\int_0^1 f(u^N_j+\theta(u^N_{j+1}-u^N_j))d\theta)
	\end{pmatrix}
	\Bigg)ds.
	\end{align*}
	This implies that 
	\begin{align*}
	\|\varepsilon_{m}\|_{\H}\leq
	&\|(E(t_m)-(\B^{-1}(h)\AA(h))^m)X^N_0\|_{\H}\\
	&+\left\|
	\sum\limits_{j=0}^{m-1}
	\int_{t_j}^{t_{j+1}}
	\left(
	E(t_{m}-s)
	-(\B^{-1}(h)\AA(h))^{(m-1-j)}
	\right)
	\begin{pmatrix}
	0\\
	P_NdW(s)
	\end{pmatrix}
	\right\|_{\H}
	\\
	&+
	\Bigg\|
	\sum\limits_{j=0}^{m-1}
	\int_{t_j}^{t_{j+1}}
	\Bigg(E(t_{m}-s)
	\begin{pmatrix}
	0\\
	P_N(f(u^N(s)))
	\end{pmatrix}
	-(\B^{-1}(h)\AA(h))^{(m-1-j)}\\
	&\quad\cdot\B^{-1}(h)
	\begin{pmatrix}
	0\\
	P_N(\int_0^1 f(u^N_j+\theta(u^N_{j+1}-u^N_j))d\theta)
	\end{pmatrix}\Bigg)ds
	\Bigg\|_{\H}\\
	\leq &Ch^{\frac \gamma 2}\|X^N_0\|_{\H^\gamma}
	+Err_m^1+Err_m^2.
	\end{align*}
	We decompose the term $Err_m^2$ into several parts as following
	\begin{align*}
	Err_m^2
	\leq&
	\Bigg\|
	\sum\limits_{j=0}^{m-1}
	\int_{t_j}^{t_{j+1}}
	(E(t_{m}-s)-E(t_{m-1}-s)\B^{-1}(h))
	\begin{pmatrix}
	0\\
	P_N(f(u^N(s)))
	\end{pmatrix}ds
	\Bigg\|_{\H}\\
	+&
	\Bigg\|
	\sum\limits_{j=0}^{m-1}
	\int_{t_j}^{t_{j+1}}
	E(t_{m-1}-s)\B^{-1}(h)
	\begin{pmatrix}
	0\\
	P_N(f(u^N(s)))
	\end{pmatrix}
	-(\B^{-1}(h)\AA(h))^{(m-1-j)}\\
	&\quad\cdot\B^{-1}(h)
	\begin{pmatrix}
	0\\
	P_N(\int_0^1 f(u^N_j+\theta(u^N_{j+1}-u^N_j))d\theta)
	\end{pmatrix}ds
	\Bigg\|_{\H}=:\MyRoman{1}+\MyRoman{2}.
	\end{align*}
	Using the H\"older inequality, the  Young inequality and \eqref{est;E_N_2}, we have
	\begin{align*}
	\MyRoman{1}
	\leq &
	\sum\limits_{j=0}^{m-1}
	\int_{t_j}^{t_{j+1}}
	\Bigg\|
	(E(h)-\B^{-1}(h))
	\begin{pmatrix}
	0\\
	P_N(f(u^N(s)))
	\end{pmatrix}\Bigg\|_{\H}
	ds\\
	\leq &
	Ch\sum\limits_{j=0}^{m-1}
	\int_{t_j}^{t_{j+1}}
	\|
	f(u^N(s))\|_{L^2}
	ds.
	\end{align*}
	For the term $\MyRoman{2},$ denoting $[\frac sh
	]$ the integer part of $\frac sh,$ we obtain
	{\small
		\begin{align*}
		&\MyRoman{2}\\
		\leq &
		\Bigg\|
		\sum\limits_{j=0}^{m-1}
		\int_{t_j}^{t_{j+1}}
		E(t_{m-1}-s)\B^{-1}(h)
		\begin{pmatrix}
		0\\
		P_N(f(u^N(s))-f(u^N(t_{[\frac sh]}))
		\end{pmatrix}ds
		\Bigg\|_{\H}\\
		+&
		\Bigg\|
		\sum\limits_{j=0}^{m-1}
		\int_{t_j}^{t_{j+1}}
		E(t_{m-1}-s)\B^{-1}(h)
		\begin{pmatrix}
		0\\
		P_N(f(u^N(t_{[\frac sh]}))-f(u^N_{[\frac sh]}))
		\end{pmatrix}ds
		\Bigg\|_{\H}\\
		+&
		\Bigg\|
		\sum\limits_{j=0}^{m-1}
		\int_{t_j}^{t_{j+1}}
		E(t_{m-1}-s)\B^{-1}(h)
		\begin{pmatrix}
		0\\
		P_N(f(u^N_{[\frac sh]})-\int_0^1 f(u^N_j+\theta(u^N_{j+1}-u^N_j))d\theta)
		\end{pmatrix}ds
		\Bigg\|_{\H}\\
		+&
		\Bigg\|
		\sum\limits_{j=0}^{m-1}
		\int_{t_j}^{t_{j+1}}
		(E(t_{m-1}-s)-(\B^{-1}(h)\AA(h))^{(m-1-j)})\B^{-1}(h)\\
		&\quad\cdot
		\begin{pmatrix}
		0\\
		P_N(\int_0^1 f(u^N_j+\theta(u^N_{j+1}-u^N_j))d\theta)
		\end{pmatrix}ds
		\Bigg\|_{\H}
		=:\MyRoman{2}_1+\MyRoman{2}_2+\MyRoman{2}_3+\MyRoman{2}_4.
		\end{align*}
	}
	Now we estimate $\MyRoman{2}_i,$ $i=1,2,3,4,$ separately.
	For the first term $\MyRoman{2}_1$, we have 
	\begin{align*}
	\MyRoman{2}_1
	\leq&
	\sum\limits_{j=0}^{m-1}
	\int_{t_j}^{t_{j+1}}
	\left\|\MM^{-1}(h)\frac{h}{2}P_N(f(u^N(s))-f(u^N(t_{[\frac sh]})))\right\|_{L^2}ds\\
	&+\sum\limits_{j=0}^{m-1}
	\int_{t_j}^{t_{j+1}}
	\left\|\MM^{-1}(h)P_N(f(u^N(s))-f(u^N(t_{[\frac sh]})))\right\|_{\dot{\H}^{-1}}ds.
	\end{align*}
	By means of the inequality $\left|\frac{2h}{4+\lambda_ih^2}\right| \leq C\lambda_i^{-\frac 12},$ where $C>1$, we have 
	\begin{align*}
	&\MyRoman{2}_1\leq C\sum\limits_{j=0}^{m-1}
	\int_{t_j}^{t_{j+1}}
	\left\|P_N(f(u^N(s))-f(u^N(t_{[\frac sh]})))\right\|_{\dot{\H}^{-1}}ds\\
	&\leq 
	C\sum\limits_{j=0}^{m-1}
	\int_{t_j}^{t_{j+1}}
	\left\|
	(u^N(s))^2+(u^N(t_{[\frac sh]}))^2+u^N(s)u^N(t_{[\frac sh]}))(u^N(s)-u^N(t_{[\frac sh]}))
	\right\|_{\dot{\H}^{-1}}ds\\
	&+
	C\sum\limits_{j=0}^{m-1}
	\int_{t_j}^{t_{j+1}}
	\left\|
	(u^N(s)+u^N(t_{[\frac sh]}))(u^N(s)-u^N(t_{[\frac sh]}))
	\right\|_{\dot{\H}^{-1}}ds\\
	&+
	C\sum\limits_{j=0}^{m-1}
	\int_{t_j}^{t_{j+1}}
	\left\|u^N(s)-u^N(t_{[\frac sh]})
	\right\|_{\dot{\H}^{-1}}ds.
	\end{align*}
	Based on the Young inequality, Sobolev embedding $L^{\frac 65}\hookrightarrow\dot{\H}^{-1}$ and the H\"older inequality,
	\begin{align*}
	\MyRoman{2}_1
	\leq &
	C\sum\limits_{j=0}^{m-1}
	\int_{t_j}^{t_{j+1}}
	(\|u^N(s)\|_{L^6}^2+\|u^N(t_{[\frac sh]})\|_{L^6}^2+1)
	\|u^N(s)-u^N(t_{[\frac sh]})\|_{L^2}ds.
	\end{align*}
	Similarly, the term $\MyRoman{2}_2$ satisfies
	$$\MyRoman{2}_2\leq
	C\sum\limits_{j=0}^{m-1}
	\int_{t_j}^{t_{j+1}}
	(\|u^N(t_{[\frac sh]})\|_{L^6}^2+\|u^N_{[\frac sh]}\|_{L^6}^2+1)
	\|u^N(t_{[\frac sh]})-u^N_{[\frac sh]}\|_{L^2}ds.
	$$
	For the term $\MyRoman{2}_3,$ we have
	\begin{align*}
	\MyRoman{2}_3
	\leq &
	\sum\limits_{j=0}^{m-1}
	\int_{t_j}^{t_{j+1}}
	\int_0^1
	\|
	\int_0^\theta f'(\lambda u^N_j+(1-\lambda)\theta(u^N_{j+1}-u^N_j))(u^N_{j+1}-u^N_j)d\lambda
	\|_{\dot{\H}^{-1}}
	d\theta ds.
	\end{align*}
	
	Then $u^N_{j+1}-u^N_{j}=\frac h2(\bar v^N_{j+1}+v^N_{j})$ and the fact that $L^{\frac 65}\hookrightarrow\dot{\H}^{-1}$ yield that
	\begin{align*}
	\MyRoman{2}_3\leq &
	Ch^2\sum\limits_{j=0}^{m-1}
	(\|u^N_j\|^2_{L^6}+\|u^N_{j+1}\|^2_{L^6}+1)
	(\|\bar v^N_{j+1}\|_{L^2}+\|v^N_{j}\|_{L^2}).
	\end{align*}
	With respect to the term $\MyRoman{2}_4,$ by using Lemma \ref{lm;prpE_n} and Lemma \ref{err;era}, we have
	\begin{align*}
	\MyRoman{2}_4\leq &
	\Bigg\|
	\sum\limits_{j=0}^{m-1}
	\int_{t_j}^{t_{j+1}}
	(E(t_{m-1}-s)-
	E(t_{m-1}-t_j))\B^{-1}(h)\\
	&\quad\cdot
	\begin{pmatrix}
	0\\
	P_N(\int_0^1 f(u^N_j+\theta(u^N_{j+1}-u^N_j))d\theta)
	\end{pmatrix}ds
	\Bigg\|_{\H}\\
	&+
	\Bigg\|
	\sum\limits_{j=0}^{m-1}
	\int_{t_j}^{t_{j+1}}
	(E(t_{m-1}-t_j)-(\B^{-1}(h)\AA(h))^{(m-1-j)})\B^{-1}(h)\\
	&\quad\cdot
	\begin{pmatrix}
	0\\
	P_N(\int_0^1 f(u^N_j+\theta(u^N_{j+1}-u^N_j))d\theta)
	\end{pmatrix}ds
	\Bigg\|_{\H}\\
	\leq &
	\sum\limits_{j=0}^{m-1}
	\int_{t_j}^{t_{j+1}}(s-t_j)\Bigg\|\B^{-1}(h)\begin{pmatrix}
	0\\
	P_N(\int_0^1 f(u^N_j+\theta(u^N_{j+1}-u^N_j))d\theta)
	\end{pmatrix}\Bigg\|_{\H^1}ds
	\\
	&+
	\sum\limits_{j=0}^{m-1}
	\int_{t_j}^{t_{j+1}}h^\frac \gamma 2
	\Bigg\|\B^{-1}(h)
	\begin{pmatrix}
	0\\
	P_N(\int_0^1 f(u^N_j+\theta(u^N_{j+1}-u^N_j))d\theta)
	\end{pmatrix}\Bigg\|_{\H^\gamma}ds
	\\
	\leq &
	C \sum\limits_{j=0}^{m-1}
	\int_{t_j}^{t_{j+1}}(s-t_j)\int_{0}^{1}
	\| f(u^N_j+\theta(u^N_{j+1}-u^N_j))\|_{L^2}d\theta ds\\
	&+
	Ch^\frac \gamma 2\sum\limits_{j=0}^{m-1}
	\int_{t_j}^{t_{j+1}}
	\int_0^1\|f(u^N_j+\theta(u^N_{j+1}-u^N_j))\|_{\dot\H^{\gamma-1}}d\theta ds.
	\end{align*}
	For the case that $d=1,\beta\geq 1,$ the Sobolev embedding $\dot\H^{\frac 12^+}\hookrightarrow L^\infty$ leads to
	\begin{equation}
\begin{split}
	\label{24estd1}
	&\int_{t_j}^{t_{j+1}}
	\int_0^1\|f(u^N_j+\theta(u^N_{j+1}-u^N_j))\|_{\dot\H^{\gamma-1}}d\theta ds\\
	\le& C\int_{t_j}^{t_{j+1}}
	\int_0^1
	(1+\|u^N_j+\theta(u^N_{j+1}-u^N_j)\|_{L^\infty}^2)\|u^N_j+\theta(u^N_{j+1}-u^N_j)\|_{\dot \H^{\gamma -1}}d\theta ds\\
	\le& Ch(1+\|u^N_j\|_{\dot \H^{{\frac 12}^+}}^2
	+\|u^N_{j+1}\|_{\dot \H^{{\frac 12}^+}}^2
	)
	(\|u^N_j\|_{\dot \H^{\gamma -1}}
	+\|u^N_{j+1}\|_{\dot \H^{\gamma -1}}).
	\end{split}
\end{equation}
	For the case that $d=2,$ $\beta=2,$ using the Sobolev embedding  $\dot\H^{1^+}\hookrightarrow L^\infty,$
	we have
\begin{equation}
	\begin{split}
	&\int_{t_j}^{t_{j+1}}
	\int_0^1\|f(u^N_j+\theta(u^N_{j+1}-u^N_j))\|
	_{\dot\H^{\gamma-1}}d\theta ds\\
	\leq &
	Ch(1+\|u^N_j\|_{\dot\H^{1^+}}^2
	+\|u^N_{j+1}\|_{\dot\H^{1^+}}^2
	)
	(\|u^N_j\|_{\dot\H^{1}}
	+\|u^N_{j+1}\|_{\dot\H^{1}})\label{24estd2}.
\end{split}	
\end{equation}
	
	For the sake of simplicity, we only take the case $d=2$ to illustrate the derivation of the strong convergent order, and omit the proof for the case that $d=1$ and $\beta>1,$ since the strong convergent order  can be obtained by replacing \eqref{24estd2} with \eqref{24estd1} for $\MyRoman{2}_4$ when $d=1.$ 
	Based on \eqref{24estd2}, we have
	\begin{align*}
	\MyRoman{2}_4
	\leq &
	Ch^\frac \gamma 2h\sum\limits_{j=0}^{m-1}
	(1+\|u^N_j\|_{\dot \H^{1^+}}^2
	+\|u^N_{j+1}\|_{\dot \H^{1^+}}^2
	)
	(\|u^N_j\|_{\dot \H^{\gamma -1}}
	+\|u^N_{j+1}\|_{\dot \H^{\gamma -1}}).
	\end{align*}
	Based on the estimates of $\MyRoman{1}$ and $\MyRoman{2}$, we obtain 
	\begin{align}\label{main-est}
	\|\varepsilon_{m}\|_{\H}\leq
	\sum\limits_{j=0}^{m-1}
	\Phi_j
	\|\varepsilon_{j}\|_{\H}+
	\sum\limits_{j=0}^{m-1}
	\psi_j+Err_m^1
	+C(h^{\frac \gamma 2}\|X^N_0\|_{\H^\gamma}),
	\end{align}
	where for $j=0,1,\cdots,(m-1),$ $\Phi_j=Ch(\|u^N(t_j)\|_{L^6}^2+\|u^N_j\|_{L^6}^2+1),$ and
	\begin{align*}
	\psi_j=
	&Ch
	\int_{t_j}^{t_{j+1}}
	\|
	(f(u^N(s)))\|_{L^2}
	ds\\
	&+C
	h^\frac \gamma 2h
	(1+\|u^N_j\|_{\dot \H^{1^+}}^2
	+\|u^N_{j+1}\|_{\dot \H^{1^+}}^2
	)
	(\|u^N_j\|_{\dot \H^{\gamma -1}}
	+\|u^N_{j+1}\|_{\dot \H^{\gamma -1}})\\
	&+
	C
	\int_{t_j}^{t_{j+1}}
	(\|u^N(s)\|_{L^6}^2+\|u^N(t_{[\frac sh]})\|_{L^6}^2+1)
	\|u^N(s)-u^N(t_{[\frac sh]})\|_{L^2}ds\\
	&+Ch^2
	(\|u^N_j\|^2_{L^6}+\|u^N_{j+1}\|^2_{L^6}+1)
	(\|\bar v^N_{j+1}\|_{L^2}+\|v^N_{j}\|_{L^2})
	.
	\end{align*}
	By the discrete Gronwall's inequality (see e.g., \cite[Lemma 2.6]{CHLZ17b}), we have
	\begin{align*}
	\|\varepsilon_{m}\|_{\H}\leq &
	\left(
	Ch^{\frac \gamma 2}\|X^N_0\|_{\H^\gamma}
	+Err_m^1
	+
	\sum\limits_{j=0}^{m-1}
	\psi_j
	\right)
	\exp
	\left(
	\sum\limits_{j=0}^{m-1}
	\Phi_j
	\right).
	\end{align*}
	Taking the $2p$th moment and using the H\"older inequality, we have
	\begin{align*}
	&\E\|\varepsilon_{m}\|_{\H}^{2p}
	\leq
	\left[\E
	\left(
	Ch^{\frac \gamma 2}\|X^N_0\|_{\H^\gamma}
	+Err_m^1
	+
	\sum\limits_{j=0}^{m-1}
	\psi_j
	\right)^{4p}
	\right]^\frac 12
	\left[\E
	\exp
	\left(
	4p\sum\limits_{j=0}^{m-1}
	\Phi_j
	\right)
	\right]^\frac 12.
	\end{align*}
	According to the exponential integrability of $u^N$ and $u^N_m,$ the above inequality becomes
	\begin{align*}
	\E\|\varepsilon_{m}\|_{\H}^{2p}
	\leq &
	Ch^{\gamma p}
	\|X^N_0\|_{\H^\gamma}^{2p}
	+
	C\left[\E
	(Err^1_m)^{4p}
	\right]^{\frac 12}
	+
	Cm^{2p-\frac 12}
	\left[
	\sum\limits_{j=0}^{m-1}
	\E
	\psi_j^{4p}
	\right]^\frac 12.
	\end{align*} 
	Thanks to the Burkholder--Davis--Gundy inequality and properties of the group in Lemma \ref{lm;prpE_n} and Lemma \ref{err;era}, we obtain
	\begin{align*}
	\E
	(Err^1_m)^{4p}
	\leq &
	\left[\int_{0}^{t_m}
	\left\|
	(E(t_{m}-s)
	-(\B^{-1}(h)\AA(h))^{(m-1-[\frac sh])}
	)
	\widetilde \G_N\circ \Q^\frac 12
	\right\|_{\mathcal L_2(\mathbb H)}^2ds
	\right]^{2p}\\
	\leq &
	C\sup_{0\leq s\leq t_m}
	\left\|
	(E(t_{m}-s)
	-E(t_{m-1}-t_{[\frac sh]})
	)
	\widetilde \G_N\circ \Q^\frac 12
	\right\|_{\mathcal L_2(\mathbb H)}^{4p}\\
	+&
	C\sup_{0\leq s\leq t_m}
	\left\|
	(E(t_{m-1}-t_{[\frac sh]})
	-(\B^{-1}(h)\AA(h))^{(m-1-[\frac sh])}
	)
	\widetilde \G_N\circ \Q^\frac 12
	\right\|_{\mathcal L_2(\mathbb H)}^{4p}\\
	\leq &
	C(h^{2\gamma p}+h^{4p}),
	\end{align*}
	which leads to
	\begin{align*}
	\E\|\varepsilon_{m}\|_{\H}^{2p}
	\leq
	Ch^{\gamma p}
	\|X^N_0\|_{\H^\gamma}^{2p}
	+
	C(h^{\gamma p}+h^{2p})
	+
	Cm^{2p-\frac 12}
	\left[
	\sum\limits_{j=0}^{m-1}
	\E
	\psi_j^{4p}
	\right]^\frac 12.
	\end{align*}
	According to the H\"older inequality, the a prior estimates of $u^N$ and $u^N_m$ and the H\"older continuity of $u^N,$ we obtain
	\begin{align*}
	\E\psi_j^{4p}\leq
	&Ch^{4p}
	\E\left(\int_{t_j}^{t_{j+1}}
	\|
	f(u^N(s))\|_{L^2}
	ds\right)^{4p}\\
	&+
	Ch^{2\gamma p}h^{4p}
	\E\left(
	(1+\|u^N_j\|_{\dot \H^{1^+}}^2
	+\|u^N_{j+1}\|_{\dot \H^{1^+}}^2
	)
	(\|u^N_j\|_{\dot \H^{\gamma -1}}
	+\|u^N_{j+1}\|_{\dot \H^{\gamma -1}})
	\right)^{4p}\\
	&+
	C\E\left(
	\int_{t_j}^{t_{j+1}}
	(\|u^N(s)\|_{L^6}^2+\|u^N(t_{[\frac sh]})\|_{L^6}^2+1)
	\|u^N(s)-u^N(t_{[\frac sh]})\|_{L^2}ds\right)^{4p}\\
	&+Ch^{8p}
	\E\left[(\|u^N_j\|^2_{L^6}+\|u^N_{j+1}\|^2_{L^6}+1)
	(\|\bar v^N_{j+1}\|_{L^2}+\|v^N_{j}\|_{L^2})
	\right]^{4p}\\
	\leq& C(h^{8p}+h^{4p+2\gamma p}+h^{8p}+h^{8p})\leq Ch^{8p}+Ch^{4p+2\gamma p}.
	\end{align*}
	This yields that
	\begin{align*}
	\E\|\varepsilon_{m}\|_{\H}^{2p}
	\leq
	Ch^{\gamma p}
	\E\|X^N_0\|_{\H^\gamma}^{2p}
	+
	C(h^{\gamma p}+h^{2p})
	+
	Cm^{2p} h^{2p+\gamma p}+Cm^{2p}h^{4p}
	\leq Ch^{\gamma p},
	\end{align*}
	which completes the proof.
\end{proof}

The above convergence result in Proposition \ref{tm;2}, together with Proposition \ref{con-spa}, implies the following strong convergence theorem.

\begin{theorem}
	\label{tm}
	Let $d=1$, $\beta\ge 1$ {\rm (}or $d=2$, $\beta=2${\rm )}, $\gamma=\min{(\beta,2)}$ and $T>0.$ 
	Assume that $X_0\in \H^\beta,$ $\|(-\Lambda)^{\frac {\beta-1}2}\Q^\frac 12\|_{\mathcal L_2(L^2)}<\infty$. There exists $h_0>0$ such that for $h\le h_0$ and $p\geq 1,$
	\begin{align*}
	\sup\limits_{m\in \Z_{M+1}}\E\left[
	(\|u(t_m)-u^N_m\|_{L^2}^2+\|v(t_m)-v^N_m\|_{\dot{\H}^{-1}}^2)^p\right]
	\leq C\left(h^{\gamma p}+\lambda_N^{-\beta p}\right),
	\end{align*}
	where $C=C(X_0,\Q,T,p)>0$, $N\in \N^+, m\in \mathbb Z_{M+1}$, $M\in \mathbb N^+, Mh=T.$
\end{theorem}

\section{Numerical experiments}
\label{sec5}
This section illustrates numerically the main results of this paper by considering the following 2-dimensional stochastic wave equation with cubic nonlinearity
\begin{equation}\label{numer_model}
\begin{split}
&du=vdt,\\
&dv=\left[u_{xx}+u_{yy}\right]dt-u^3dt+dW, \quad\quad \;\; \text{in} \;\; \OO \times (0,T],\\
&u(0)=0, \quad v(0)=1, \quad \;\; \text{in} \;\; \OO
\end{split}
\end{equation}
with $\OO=(0,1)\times(0,1),$ $T=1$ and homogenous Dirichlet boundary condition.

In the sequel, we choose the orthonormal basis $\{e_{k,l}\}_{k,l\in \N^+}$ and the corresponding eigenvalues $\{\eta_{k,l}\}_{k,l\in \N^+}$ of $\Q$ as
\begin{equation*}
e_{k,l}=2\sin(k\pi x)\sin(l\pi y),\quad\quad \eta_{k,l}=\frac{1}{k^3+l^3}.
\end{equation*}

\begin{figure}[h]
	\centering
	\subfigure[${\rm Energy-preserving}$]{
		\begin{minipage}{0.3\linewidth}
			\includegraphics[width=3.5cm,height=4cm]{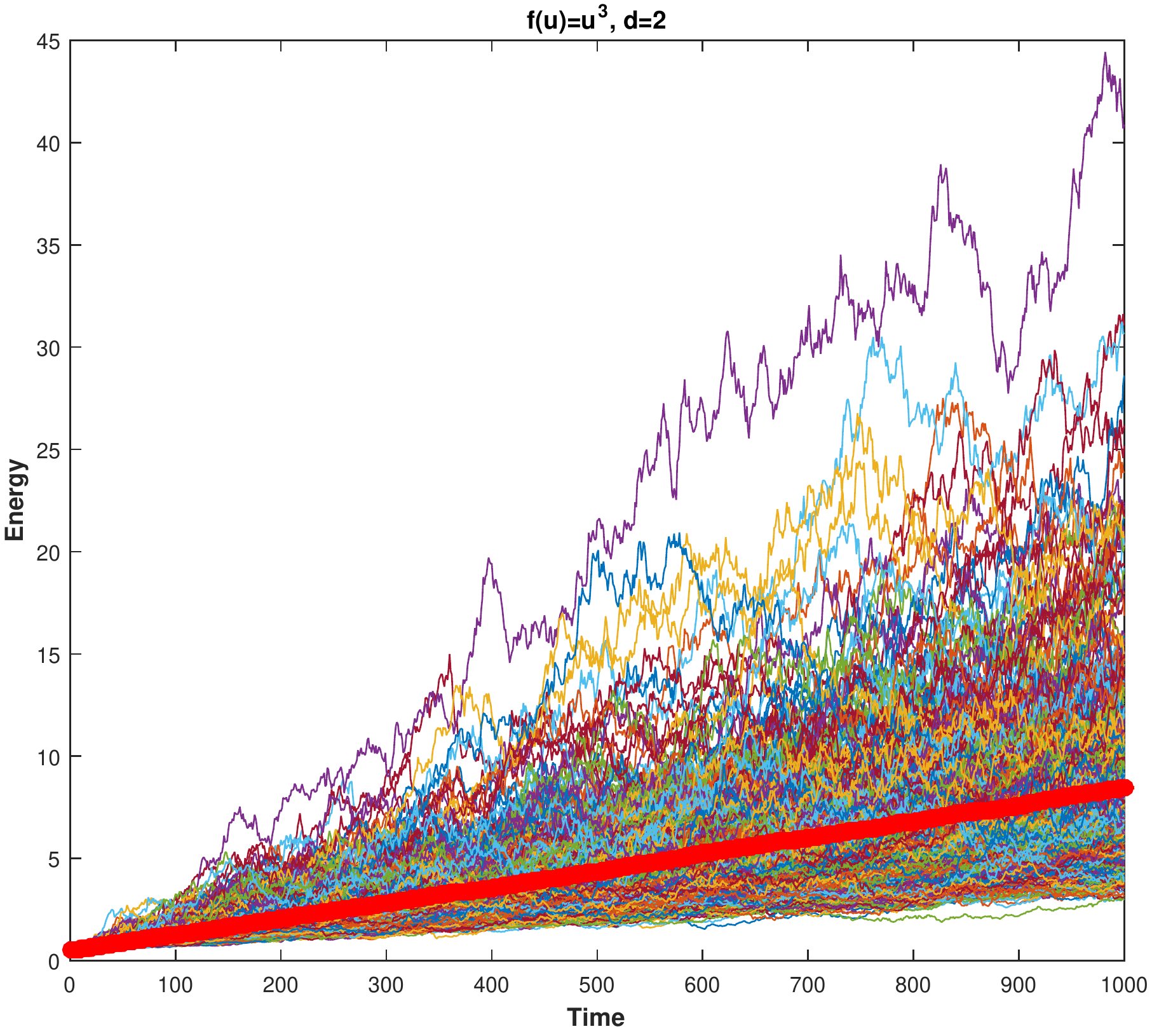}
	\end{minipage}}
	\subfigure[${\rm Spatial\; error}$]{
		\begin{minipage}{0.3\linewidth}
			\includegraphics[width=3.5cm,height=4cm]{{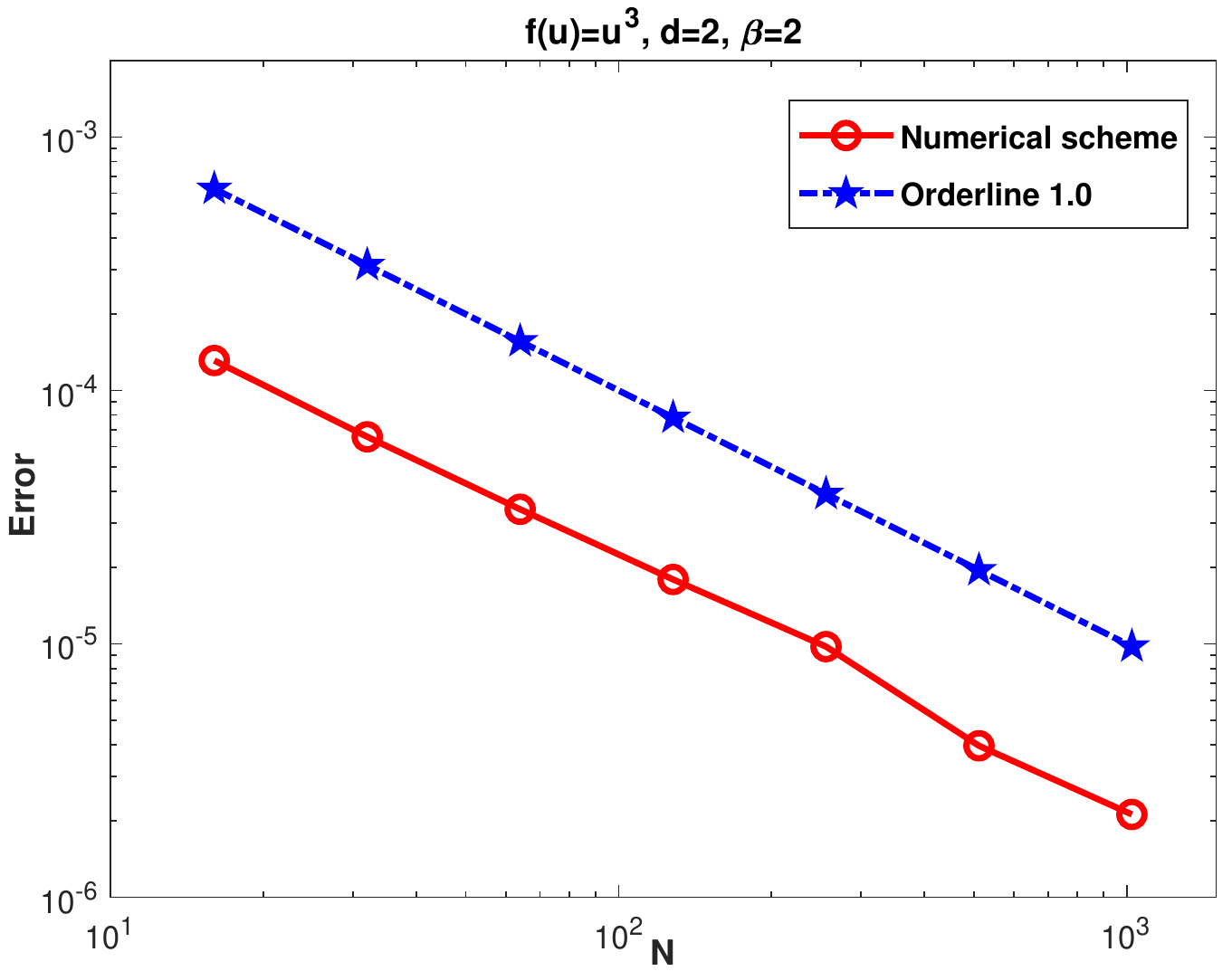}}
	\end{minipage}}
	\subfigure[${\rm Temporal\; error}$]{
		\begin{minipage}{0.3\linewidth}
			\includegraphics[width=3.5cm,height=4cm]{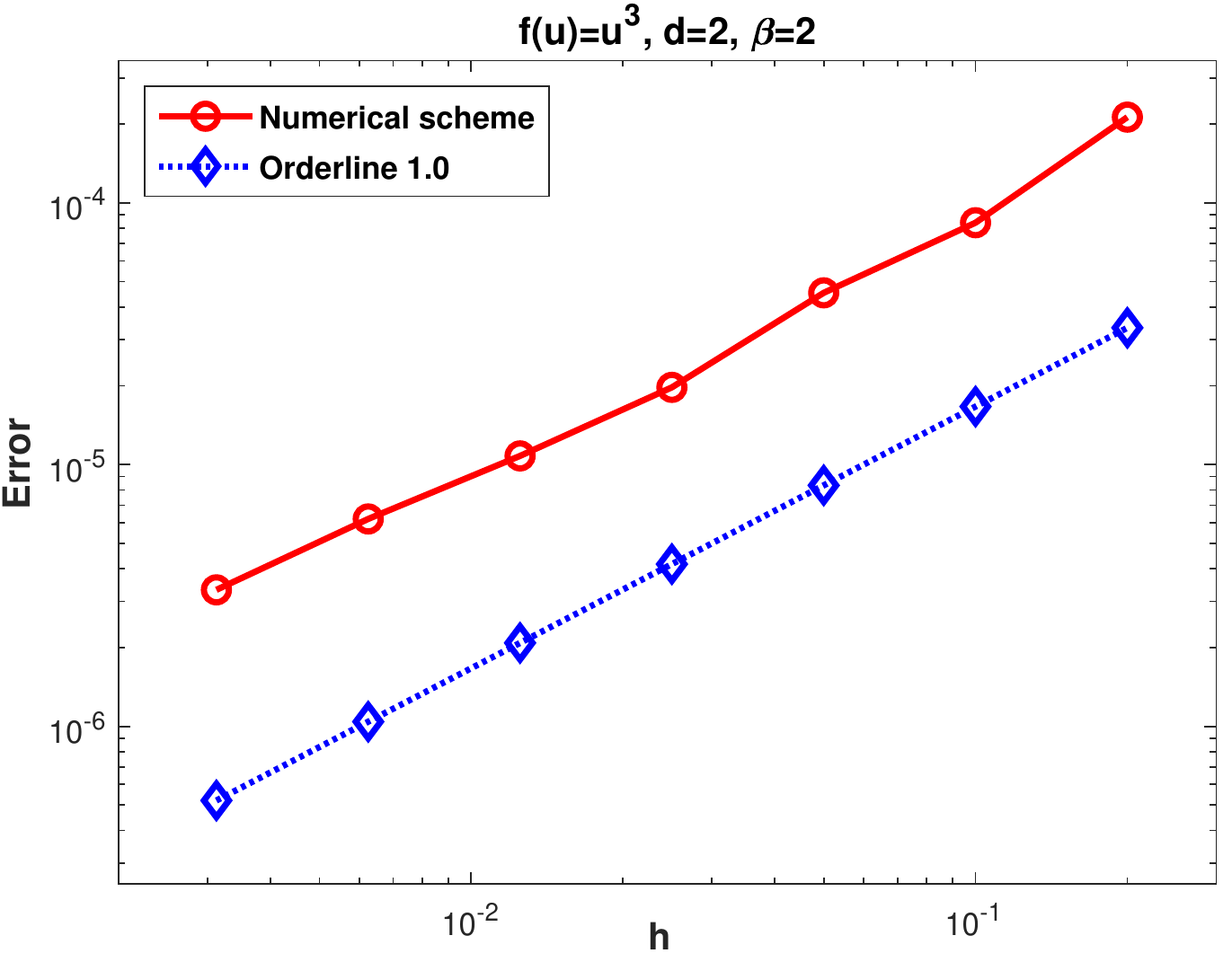}
	\end{minipage}}
	\caption{Energy-preserving property and strong convergence order in space and time for the full discrete numerical method \eqref{sche;sbes} applied to stochastic wave equation \eqref{numer_model}.}\label{temporal_order}
\end{figure}

Fig. \ref{temporal_order}(a) presents the evolution of discrete energy using the proposed methods in Section \ref{sec4}, where the red line represents the discrete averaged energy along 500 trajectories. 
From Theorem 4.1, we noted that, the averaged energy evolution law $\mathbb E(V_1(u_m^N,v_m^N))$ follows a linear evolution with growth rate $\frac{1}{2}\tr((P_N\Q^{\frac12})(P_N\Q^{\frac12})^*)$. It can be observed from Fig. \ref{temporal_order}(a) that the discrete averaged energy obeys nearly linear growth over 500 trajectories, which coincides with the theoretical analysis.

Now let us start with tests on the convergence rates. First of all, we consider the spatial convergence rate of the numerical method \eqref{sche;sbes}. 
The middle figure in Fig. \ref{temporal_order} displays the spatial approximation errors $\sup\limits_{0\leq t\leq T}\|X_N(t)-X(t)\|_{L^p(\OOO; \H)}$ against $N$ on a log-log scale with $N=2^s,~s=4,5,\cdots,9$. 
It can be observed that the slope is closed to $1$ which is consistent with our previous theoretical result (see Proposition \ref{con-spa}) on the spatial convergence order. Note that for the temporal discretization we used here the proposed method \eqref{sche;sbes} at a sufficiently small time step-size $h=2^{-10}$. In addition, $N=2^{11}$ is used to simulate the exact solution.

To investigate the strong convergence order in temporal direction of \eqref{sche;sbes} by using various step-sizes $h=2^{-r},~r=2,3,\cdots,7$, we now fix $N=100$. Again, the ``exact'' solution is approximated by the method \eqref{sche;sbes} with a very small time step-size $h=2^{-12}$.
The right figure in Fig. \ref{temporal_order} presents the strong approximation errors of the proposed method \eqref{sche;sbes} in temporal direction.
It can be seen that this numerical performance coincides with the theoretical assertion (see Proposition \ref{tm;2}).

\vspace{1em}

\noindent{\bf Acknowledgment.} The authors are very grateful to the referees for their valuable comments and suggestions on the improvement of the present paper.

\section{Appendix}
{\bf{Existence and uniqueness of numerical solution of \eqref{sche;sbes}}.} 
To prove the well-posedness of the numerical scheme, it only needs to show that the existence and uniqueness of numerical solution of the scheme at the step $m+1$.  
Assume that  $d\le 2,$  $\|\Q^\frac 12\|_{\mathcal L_2(L^2)}<\infty,$ $h_0>0$ is sufficient small, temporal step size $h<h_0$, and that $(u_m^N,v_m^N)^\top \in\H^1$ is $\mathcal F_{t_m}$-measurable and $\E\Big[\exp({V_1(u_m^N,v_m^N)}{e^{-\alpha t_m}})\Big]\le C$ for some $\alpha(Q,T,u_0,v_0,d) \ge1$ and $C(Q,T,u_0,v_0,d)>0$. 
The proof is similar to that of \cite[Lemma 8.1]{MSH02} by making use of the structure of the drift coefficient $f$ and the fact that \eqref{mod;swe} is a separable system.

Since the last equation in \eqref{sche;sbes} has an explicit analytical solution, it suffices to show the existence of a unique solution for the first two equations in \eqref{sche;sbes}, that is,
\begin{align}\label{sche;sbes;det}
u^N_{m+1}=&u^N_{m}+h\bar v^N_{m+\frac 12},\\\nonumber
\bar v^N_{m+1}=&v^N_m+h\Lambda_N u^N_{m+\frac 12}-hP_N\left(\int_0^1f(u^N_m+\theta(u^N_{m+1}-u^N_m))d\theta \right).
\end{align}
Finding the solution of the above system is equivalent to the solvability of finding $\widetilde v$ such that
\begin{align*}
\widetilde v-v^N_{m}-\frac 12 h\Lambda_N (u^N_{m}+\frac 12h \widetilde v)+\frac 12h P_N\left(\int_0^1f(u^N_m+\theta h \widetilde v)d\theta \right)=0.
\end{align*}
Choosing $\widetilde F(x,y)>0$ as the anti-derivative of  $\frac 12 h\int_0^1f(x+\theta h y)d\theta$ for fixed $x$, then it suffices to find $v^*$ which minimizes
\begin{align*}
\frac 12\|\widetilde v-v^N_{m}\|_{L^2}^2+ \frac 12\|u^N_{m}+\frac 12h \widetilde v\|_{\mathbb {\dot{\mathbb H}}^1}^2
+\int_{\mathcal O}\widetilde F(u^N_m, \widetilde v) dx. 
\end{align*}
The existence of the minimizer is guaranteed by the fact that $f(\cdot)$ is smooth and its anti-derivative is bounded below. Furthermore, the uniqueness can be obtained by using the fact $f(u)=c_3u^3+\cdots+c_1u+c_0$ is a polynomial 
with $c_3>0$. It implies that there exists $C(c_3,c_2,c_1,c_0)>0$ such that 
$(f(x)-f(y))\cdot(x-y)\ge -C\|x-y\|^2.$

Assume that we have two different numerical solutions $(\widetilde u_1,\widetilde v_1)$ and $(\widetilde u_2, \widetilde v_2)$ in $P_N(\mathbb H)$ of \eqref{sche;sbes} with the same initial condition $(u_m^N,v_m^N)$. 
Then 
\begin{align*}
\widetilde v_i-\frac 12h\Lambda_N(u_m^N+\frac h2\widetilde v_i)+ P_N\left(\frac{\partial }{\partial y} \widetilde F(u^N_m, \widetilde v_i)\right)=v_m^N,\; i=1,2.
\end{align*}
Using $c_3>0$ and the Poincar\'e inequality, we have that 
\begin{align*}
0\ge (1+\lambda_1\frac {h^2}4-C' h^2)\|\widetilde v_1-\widetilde v_2\|_{L^2}^2
\end{align*}
for some constant $C'(c_3,c_2,c_1,c_0).$ By  taking $h_0$ small enough such that $1+ \lambda_1 \frac {h_0^2}4> C' h_0^2$, the uniqueness of the numerical solution of \eqref{sche;sbes} is obtained.

{\noindent{\bf {Strong approximate error of the fixed point iteration.}}} 
{\color{black}Assume that $(u_m^N,v_m^N)$ is the numerical solution at time step $m.$
The first fixed point iteration reads 
\begin{equation}
\begin{split}
\label{iteration1}
&u^{N,k+1}=u^N_{m}+h\bar v^{N,k+\frac 12},\\
&\bar v^{N,k+1}=v^N_m+h\Lambda_N u^{N,k+\frac 12}-hP_N\Big(\frac 12((u^{N,k})^2+(u^{N}_m)^2)-1\Big)u^{N,k+\frac 12},\\
&u^{N,0}=u^N_{m},\quad v^{N,0}=v^N_{m},
\end{split}
\end{equation}
where $u^{N,k+\frac 12}=\frac 12(u^{N,k+1}+u^{N}_{m}),$ $\bar v^{N,k+\frac 12}=\frac 12(\bar v^{N,k+1}+v^{N}_{m}).$ 
The second one takes the form of
\begin{equation}
\begin{split}
\label{iteration2}
u^{N,k+1}=&u^N_{m}+h\bar v^{N,k+\frac 12},\\
\bar v^{N,k+1}=&v^N_m+h\Lambda_N u^{N,k+\frac 12}\\
&-h \mathbb I_{\{\|u^{N,k}\|_{\dot \H^1}+\|u^N_m\|_{\dot \H^1} \le \frac 1\epsilon\}} P_N\Big(\frac 12\Big((u^{N,k})^2+(u^{N}_m)^2\Big)-1\Big)u^{N, k+\frac 12}
\end{split}
\end{equation}
with $u^{N,0}=u^N_{m}, v^{N,0}=v^N_{m}$ and a sufficiently small parameter $\epsilon >0$. 
Under the assumption on the well-posedness of numerical solution  in Appendix, we obtain that  \eqref{iteration1} satisfies for small $h\in (0,h_0)$ with $h^2\lambda_N^{\frac d 3}=O(1),$
\begin{align}\label{err-est}
&\E\Big[\|u^{N,k}-u^*\|_{L^2}^2\Big]+\E\Big[\|v^{N,k}-v^*\|_{\dot {\mathbb H}^{-1}}^2\Big]\nonumber\\
\le& Ch^2(1+h^2\lambda_N^{\frac {d}3})^{ 2} \left((\frac 12)^{k+1}+\exp(-\eta h^{-1}\min(h^{-2}\lambda_N^{-\frac d3},1))\right),
\end{align}
and \eqref{iteration2} satisfies for small $h\in (0,h_0)$ with $\frac h{\epsilon^4}= O(1),$
\begin{align}\label{err-est1}
&\E\Big[\|u^*-u^{N,k+1}\|_{L^2}^2\Big]+\E\Big[\|v^{N,k}-v^*\|_{\dot {\mathbb H}^{-1}}^2\Big]\nonumber\\
\le& C\max(h^2,\epsilon^8)\Big((\frac 12)^{k+1}+\exp(-\eta \frac 1{\epsilon^2})\Big),
\end{align} 
where $(u^*,v^*)$ is the fixed point of \eqref{sche;sbes;det}, i.e., $(u^*,v^*)=(u_{m+1}^N,\bar v_{m+1}^N)$, and $\eta,C>0$ depend on $T,Q,d,u^N_m,v^N_m.$ Here $h_0$ is the upper bound of the step-size such that the fixed point of \eqref{sche;sbes;det} exists uniquely.

{\bf{$\bullet$ Analysis of fixed point iterations of \eqref{iteration1}}}

Since the stochastic subsystem in \eqref{sche;sbes} {\color{black}possesses} a unique explicit solution, it suffices to construct iterations to approximate the fixed point of \eqref{sche;sbes;det}.
Thanks to the polynomial assumption on $f$, we have that 
\begin{align*}
\int_0^1f(u^N_m+\theta(u^N_{m+1}-u^N_m))d\theta
&=\frac 14 \frac {(u^N_{m+1})^4-(u^N_{m})^4}{u^N_{m+1}-u^N_m}-\frac 12 \frac {(u^N_{m+1})^2-(u^N_{m})^2}{u^N_{m+1}-u^N_m}\\
&=\Big(\frac 12((u^N_{m+1})^2+(u^N_m)^2)-1\Big) u^{N}_{m+\frac 12},
\end{align*}
which yields 
\begin{equation}
\begin{split}
\label{equ-fix1}
\frac{u^{N}_m+u^*}2=&u^N_{m}+\frac 12hv^N_m +\frac 14 h^2\Lambda_N \left(\frac{u^{N}_m+u^*}2\right)\\
&-\frac 14h^2P_N\Big(\frac 12((u^*)^2+(u^{N}_m)^2)-1\Big)\left(\frac{u^{N}_m+u^*}2\right).
\end{split}
\end{equation}
Moreover, we rewrite \eqref{iteration1} as
\begin{equation}
\begin{split}
\label{equ-fix}
u^{N,k+\frac 12}=&u^N_{m}+\frac 12hv^N_m +\frac 14 h^2\Lambda_N u^{N,k+\frac 12}\\
&-\frac 14h^2P_N\Big(\frac 12((u^{N,k})^2+(u^{N}_m)^2)-1\Big)u^{N,k+\frac 12}.
\end{split}
\end{equation}
To prove \eqref{err-est}, we estimate  
	\begin{align*}
	\E\Big[\|u^{N,k+\frac 12}-\frac 12(u^*+u^N_m)\|_{L^2}^2\Big]+\E\Big[\|{\bar v}^{N,k+\frac 12}-\frac 12(v^*+v^N_m)\|_{\dot \H^{-1}}^2\Big].
	\end{align*} 
	Let us take  $\E\Big[\|u^{N,k+\frac 12}-\frac 12(u^*+u^N_m)\|_{L^2}^2\Big]$ as an example to illustrate the detailed steps and estimates, since the estimate of $\E\Big[\|{\bar v}^{N,k+\frac 12}-\frac 12(v^*+v^N_m)\|_{\dot \H^{-1}}^2\Big]$ is analogous.

First, we claim that the a priori estimate of $u^{N,k+\frac 12}$ could be bounded by the estimates of $u^N_m$ and $v^N_m$.   
To this end, we take $L^2$-inner product with $u^{N,k+\frac 12}$ on \eqref{equ-fix} and use  the integration by part formula and Young's inequality, then
	\begin{align*}
	\|u^{N,k+\frac 12}\|_{L^2}^2\le&  \<u^N_{m}+\frac 12hv^N_m,u^{N,k+\frac 12}\>_{L^2} +\frac 14 h^2\<\Lambda_N u^{N,k+\frac 12},u^{N,k+\frac 12}\>_{L^2}\\\nonumber
	&-\frac 14h^2\<\Big(\frac 12((u^{N,k})^2+(u^{N}_m)^2)-1\Big)u^{N,k+\frac 12},u^{N,k+\frac 12}\>_{L^2}\\
	\le& \frac 12\|u^N_{m}+\frac 12hv^N_m\|_{L^2}^2+\|u^{N,k+\frac 12}\|_{L^2}^2(\frac 12+ \frac 14 h^2 )-\frac 14h^2\|\nabla u^{N,k+\frac 12}\|_{L^2}^2.
	\end{align*}
	
	Then since $\lambda_1 \ge 1,$ by the Poincar\'e inequality, we have that 
	\begin{align*}
	\frac 12\|u^{N,k+\frac 12}\|_{L^2}^2&\le \frac 12 \|u^N_{m}+\frac 12 hv^N_m\|_{L^2}^2.
	\end{align*}
	Similarly, one could get that for $q\in \mathbb N^+,$
	\begin{align}\label{L6-est}
	\|u^{N,k+\frac 12}\|_{L^{2q}}^{2q}&\le  C_q \|u^N_{m}+\frac 12 hv^N_m\|_{L^{2q}}^{2q}.
	\end{align}
	By subtracting the corresponding equation {\eqref{equ-fix1}} of $\frac{u_m^N+u^*}2$ from \eqref{equ-fix}, we achieve that 
	\begin{align*}
	u^{N,k+\frac 12}-\frac  {u_m^N+u^*}2=& \frac 14 h^2\Lambda_N( u^{N,k+\frac 12}-\frac  {u_m^N+u^*}2)\nonumber\\\nonumber 
	&-\frac 14h^2P_N\Big(\frac 12((u^{N,k})^2+(u^{N}_m)^2)-1\Big)(u^{N,k+\frac 12}-\frac  {u_m^N+u^*}2)\\
	&-\frac 14h^2P_N\Big(\frac 12(u^{N,k})^2-\frac 12(u^*)^2\Big)\frac  {u_m^N+u^*}2.
	\end{align*}
	Then multiplying $u^{N,k+\frac 12}-\frac  {u_m^N+u^*}2$ on both sides of the above equation, using integration by parts, the Sobolev embedding theorem $L^{\frac 65}\hookrightarrow\dot{\mathbb H}^{-\frac d3}$ and Young's inequality, we have that  for $k\ge 1,$
	\begin{align*}
	&\|u^{N,k+\frac 12}-\frac  {u_m^N+u^*}2\|_{L^2}^2\nonumber\\	\le&-\frac 14 h^2  \left\|(-\Lambda)^{\frac 12}\left(u^{N,k+\frac 12}-\frac  {u_m^N+u^*}2\right)\right\|_{L^2}^2 +\frac 14 h^2  \|u^{N,k+\frac 12}-\frac  {u_m^N+u^*}2\|_{L^2}^2\nonumber\\\nonumber
	&-\frac 14h^2P_N\left\<(-\Lambda)^{-\frac 12}	\left(\Big(\frac 12(u^{N,k})^2-\frac 12(u^*)^2\Big)\frac{u_m^N+u^*}2\right),(-\Lambda)^{\frac 12}\big(u^{N,k+\frac 12}-\frac  {u_m^N+u^*}2\big)	\right\>_{L^2}\\
	\le& \frac 14 h^2  \|u^{N,k+\frac 12}-\frac  {u_m^N+u^*}2\|_{L^2}^2\\\nonumber
	&+ h^2 C\Big(\|u^{N,k}\|_{L^6}^4+\|u_m^N\|_{L^6}^4+\|u^{*}\|_{L^6}^4\Big) \left\|u^{N,k-\frac 12}-\frac  {u_m^N+u^*}2\right\|_{L^2}^2,
	\end{align*} 
where $C$ depends on the coefficient of  Sobolev embedding inequality $L^{\frac 65}\hookrightarrow\dot{\mathbb H}^{-\frac d3}$.  	
This implies that 
	\begin{align}
	&\|u^{N,k+1}-u^*\|_{L^2}^2\nonumber\\
	\le& \frac {Ch^2}{1-\frac 14 h^2}\Big(\|u^{N,k}\|_{L^6}^4+\|u_m^N\|_{L^6}^4+\|u^{*}\|_{L^6}^4\Big) \|u^{N,k}-u^*\|_{L^2}^2\label{l2-est0}\\\nonumber
	\le& (\frac {Ch^2}{1-\frac 14 h^2})^{k+1}\prod_{j=0}^{k}\Big(\|u^{N,j}\|_{L^6}^4+\|u_m^N\|_{L^6}^4+\|u^{*}\|_{L^6}^4\Big) \|u^{N,0}-u^*\|_{L^2}^2.
	\end{align}
	One could use the Sobolev embedding theorem {$\dot\H^{\frac d3}\hookrightarrow  L^6$}, the inverse inequality in $U^N$ and \eqref{L6-est}, and get for small $h\in (0,1)$,
	\begin{align*}
	&\|u^{N,k+1}-u^*\|_{L^2}^2\\
	\le& (\frac {C' h^2}{1-\frac 14 h^2})^{k+1} \prod_{j=0}^{k}\Big(\lambda_N^{\frac {2d}3}h^4\|v^{N}_m\|_{L^2}^4+\|u^N_m\|_{L^6}^4+\|u^{*}\|_{L^6}^4\Big) \|u^{N,0}-u^*\|_{L^2}^2,
	\end{align*}
	where $C'=C\max(1,C_{sob}^4),$ where $C_{sob}$ is the coefficient of the  Sobolev embedding $\dot\H^{\frac d3}\hookrightarrow  L^6.$
	Furthermore, on the subset $$\Omega_h:=\Big\{\omega\in \Omega \Big| \lambda_N^{\frac {2d}3}h^4\|v^{N}_m\|_{L^2}^4+\|u^N_m\|_{L^6}^4+\|u^{*}\|_{L^6}^4 \le \frac {1-\frac 14 h^2}{2 C' h^2}\Big\},$$  it holds that for 
	\begin{align*}
	\E\Big[\mathbb I_{\Omega_h}\|u^{N,k+1}-u^*\|_{L^2}^2\Big]&\le \left(\frac 12\right)^{k+1}\E\Big[\|u^{N,0}-u^*\|_{L^2}^2\Big]\\
	&\le \left(\frac 12\right)^{k+1} h^2 \E\Big[\|\frac {v_m^N+v^*}2 \|_{L^2}^2\Big]\\
	&\le C(Q,T,u_m^N,v_{m}^N,d) h^2\left(\frac 12\right)^{k+1}.
	\end{align*} 
	Notice that the fixed point of \eqref{sche;sbes;det} in the {\color{black}revised paper} possesses the energy preserving property $V_1(u^*,v^*)=V_1(u^N_m,v^N_m).$
	The first equality of \eqref{l2-est0}, together with \eqref{L6-est}, the moment boundedness of $V_1(u^N_m,v^N_m)$ and $V_1(u^*,v^*)$, and the inverse Sobolev inequality, yield that for $p\ge 1$,
	\begin{align}\label{det-est}
	&\quad\E\Big[\|u^{N,k+1}-u^*\|_{L^2}^{2p}\Big]\\\nonumber
	&\le (\frac {Ch^2}{1-\frac 14 h^2})^p  \Big(\E\Big[\Big(\|u^{N,k}\|_{L^6}^{8p}+\|u_m^N\|_{L^6}^{8p}+\|u^{*}\|_{L^6}^{8p}\Big)\Big]\Big)^{\frac 12} \Big(\E\Big[\|u^{N,k-\frac 12}-\frac  {u_m^N+u^*}2\|_{L^2}^{4p}\Big]\Big)^{\frac 12}\\\nonumber
	&\le (\frac {Ch^2}{1-\frac 14 h^2})^p   \Big(1+\Big(\E\Big[\|u^{N}_m\|_{L^6}^{8p}\Big]\Big)^{\frac 12}+(h\lambda_N^{\frac {d}6})^{4p}\Big(\E\Big[\|v^{N}_m\|_{L^2}^{8p}\Big]\Big)^{\frac 12}\Big)\\\nonumber 
	&\quad\times \Big(\E\Big[\|u^{N,k}\|_{L^2}^{4p}+\|u_m^N\|_{L^2}^{4p}+\|u^*\|_{L^2}^{4p}\Big]\Big)^{\frac 12} \\\nonumber
	&\le C(T,Q,p,d,u^N_m,v^N_m)(h(1+h^2\lambda_N^{\frac {d}3}))^{2p},
	\end{align}
where $C$ differs from line to line up to a constant depending on $p,$ and $C(T,Q,p,d,u^N_m,v^N_m)<\infty$ due to $\eqref{L6-est}$ and the energy preserving property of the fixed point.
	On the set $\Omega_h^c$, by \eqref{det-est} and the Chebshev inequality,  we have that there exists $C>0$ depending on $T,Q,u^N_m,v^N_m,d$ such that
	\begin{align*}
	&\E\Big[\mathbb I_{\Omega_h^c}\|u^{N,k+1}-u^*\|_{L^2}^2\Big]\\
	\le& \Big(\mathbb P(\lambda_N^{\frac {2d}3}h^4\|v^{N}_m\|_{L^2}^4+\|u^N_m\|_{L^6}^4+\|u^{*}\|_{L^6}^4 \ge \frac {1-\frac 14 h^2}{2 C' h^2} )\Big)^{\frac 12}\\
	&\quad\times \sqrt{\E\Big[\|u^{N,k+1}-u^*\|_{L^2}^4\Big]}\\
	\le& Ch^2(1+h^2\lambda_N^{\frac {d}3})^{2} \Big(\mathbb P(\lambda_N^{\frac {2d}3}h^4\|v^{N}_m\|_{L^2}^4+\|u^N_m\|_{L^6}^4+\|u^{*}\|_{L^6}^4 \ge \frac {1-\frac 14 h^2}{2 C' h^2} )\Big)^{\frac 12}.
	\end{align*}

Now applying the Gaussian type tail estimates for $(u_m^N,v_m^N)$ and $(u^*,v^*)$ which possess some exponential integrability (see e.g. the proof of [13, Corollary 4.1]), and using the Sobolev embedding theorem, the assumption $$\E\Big[\exp({V_1(u_m^N,v_m^N)}{e^{-\alpha t_m}})\Big]< \infty$$ in Appendix, as well as the definition of $V_1$,
we get that
 \begin{align*}
 &\mathbb P\Big(\lambda_N^{\frac {2d}3}h^4\|v^{N}_m\|_{L^2}^4+\|u^N_m\|_{L^6}^4+\|u^{*}\|_{L^6}^4 \ge \frac {1-\frac 14 h^2}{2 C' h^2}\Big)\\
 \le& \mathbb P\Big(\|v^{N}_m\|_{L^2}\ge h^{-1}\lambda_N^{-\frac d6}(\frac 13 \frac {1-\frac 14 h^2}{2 C' h^2})^{\frac 14}\Big)+
  \mathbb P\Big( \|u^N_m\|_{\dot{\mathbb H}^{1}} \ge (\frac {C_{sob}^4}3 \frac {1-\frac 14 h^2}{2 C' h^2})^{\frac 14}\Big)\\
&+ \mathbb P\Big( \|u^{*}\|_{\dot{\mathbb H}^{1}} \ge (\frac {C_{sob}^4}3 \frac {1-\frac 14 h^2}{2 C' h^2})^{\frac 14}\Big)\\
\le& C_{11} \exp(-\eta_1 (h^{-2}\lambda_N^{-\frac d3}(\frac 13 \frac {1-\frac 14 h^2}{2 C' h^2})^{\frac 12}))+C_{12}\exp(-\eta_2 (\frac {C_{sob}^4}3 \frac {1-\frac 14 h^2}{2 C' h^2})^{\frac 12})\\
&+C_{13}\exp(-\eta_3 (\frac {C_{sob}^4}3 \frac {1-\frac 14 h^2}{2 C' h^2})^{\frac 12}),
\end{align*}
where $C_{sob}$ is the coefficient of Sobolev embedding inequality from $\dot{\mathbb H}^{1}$ to $L^6,$
\begin{align*}
C_{11}&=\E[\exp(\frac 12\|v_m^N\|_{L^2}^2e^{-\alpha {t_m}})], C_{12}=C_{13}=\E[\exp(\frac 12\|u_m^N\|_{{\dot{\mathbb H}^1}}^2e^{-\alpha {t_m}})], \\
\eta_1&=\eta_2=\eta_3=\frac 12e^{-\alpha {t_m}}. 
\end{align*}
Therefore, we  have  that 
\begin{align*}
&\quad\E\Big[\mathbb I_{\Omega_h^c}\|u^{N,k+1}-u^*\|_{L^2}^2\Big]\\
&\le Ch^2(1+h^2\lambda_N^{\frac {d}3})^{2} \Big(\mathbb P(\lambda_N^{\frac {2d}3}h^4\|v^{N}_m\|_{L^2}^4+\|u^N_m\|_{L^6}^4+\|u^{*}\|_{L^6}^4 \ge \frac {1-\frac 14 h^2}{2 C' h^2} )\Big)^{\frac 12}\\
&\le Ch^2(1+h^2\lambda_N^{\frac {d}3})^{2} (C_{11}^{\frac 12}+C_{12}^{\frac 12}+C_{13}^{\frac 12})\times \\
&\max(\exp(-\eta_1 (h^{-2}\lambda_N^{-\frac d3}(\frac 13 \frac {1-\frac 14 h^2}{2 C' h^2})^{\frac 12})),\exp(-\eta_3 (\frac {C_{sob}^4}3 \frac {1-\frac 14 h^2}{2 C' h^2})^{\frac 12})).
\end{align*}

Choosing $h\in (0,1)$ small enough,  we conclude that there exist $\eta,C>0$ depending on $Q,T,u_m^N,v_m^N,d$ such that
\begin{align*}
&\E\Big[\mathbb I_{\Omega_h^c}\|u^{N,k+1}-u^*\|_{L^2}^2\Big]\\
\le& Ch^2(1+h^2\lambda_N^{\frac {d}3})^{2}\exp(-\eta h^{-1}\min(h^{-2}\lambda_N^{-\frac d3},1)).
\end{align*}

	Combining the above two estimates on $\Omega_h$ and $\Omega_h^c$, we complete the proof.

{\bf{$\bullet$Analysis of fixed point iterations of \eqref{iteration2}}}

	Using the mild form of \eqref{iteration2} and the unitary property of $\B^{-1}(h)\AA(h)$ in Lemma 4.1 and the expression of $\B^{-1}(h)$, we have that  for $k\in \mathbb  N^+$,
	\begin{align*}
	&\sqrt{\|u^{N,k+1}\|_{\dot \H^1}^2+\|\bar v^{N,k+1}\|_{L^2}^2}\\\nonumber
	\le& \sqrt{\|u^N_{m}\|_{\dot \H^1}^2+\| v^{N}_m\|_{L^2}^2}+ hC\Big(\|u^{N,k+\frac 12}\|_{L^2}+\|u^{N,k+\frac 12}\|_{L^6}(\|u^{N,k}\|_{L^6}^2+\|u^N_m\|_{L^6}^2)\Big)\\\nonumber
	&\times \mathbb I_{\{\|u^{N,k}\|_{\dot \H^1}+\|u^N_m\|_{\dot \H^1} \le \frac 1\epsilon \}} \\\nonumber
	\le&  \sqrt{\|u^N_{m}\|_{\dot \H^1}^2+\| v^{N}_m\|_{L^2}^2}+
	hC(1+\frac 1{\epsilon^2})\sqrt{\|u^{N,k+1}\|_{\dot \H^1}^2+\|\bar v^{N,k+1}\|_{L^2}^2}+
	Ch(\frac 1 \epsilon+\frac 1 {\epsilon^3}).
	\end{align*}
	This implies that if $hC(1+\frac 1{\epsilon^2})\leq c<1,$ then there exists $C'>0$ such that 
	\begin{align}\label{pri-est-fix}
	&\sqrt{\|u^{N,k+1}\|_{\dot \H^1}^2+\|\bar v^{N,k+1}\|_{L^2}^2}\\\leq&
	(1+C'h(1+\frac 1{\epsilon^2}))\sqrt{\|u^N_{m}\|_{\dot \H^1}^2+\| v^{N}_m\|_{L^2}^2}+(1+C'h(1+\frac 1{\epsilon^2}))Ch(\frac 1 \epsilon+\frac 1 {\epsilon^3}).\nonumber
	\end{align}
	Let us consider the difference between $\frac {u^*+u^N_m}2$ and $u^{N,k+\frac 12}$. The estimate with respect to $v^*$ and $\bar v^{N,k+1}$ is similar.
	By substituting the equation of $\bar v^{N,k+1}$ into the equation of $ u^{N,k+1}$, we have that 
	\begin{align*}
	u^{N,k+\frac 12}=&u^N_m+\frac h2v^N_m+\frac {h^2} 4 \Lambda_N u^{N,k+\frac 12}-\frac {h^2}4 P_N\Big(\frac 12\Big((u^{N,k})^2+(u^{N}_m)^2\Big)-1\Big)u^{N, k+\frac 12}\\
	&\times  \mathbb I_{\{\|u^{N,k}\|_{\dot \H^1}+\|u^N_m\|_{\dot \H^1} \le \frac 1\epsilon\}}.
	\end{align*}
	By subtracting the above equation from the equation of $\frac {u^*+u^N_m}2$,  taking $L^2$-inner product with $\frac {u^*+u^N_m}2-u^{N,k+\frac 12}$ and using Young's inequality, we obtain 
	\begin{align*}
	&\|\frac {u^*+u^N_m}2-u^{N,k+\frac 12}\|_{L^2}^2\\
	\le& -\frac {h^2} 8 \|\nabla(\frac {u^*+u^N_m}2- u^{N,k+\frac 12})\|_{L^2}^2+\frac {h^2}4\|\frac {u^*+u^N_m}2-u^{N,k+\frac 12}\|_{L^2}^2\\
	&+\frac {h^2}8 \|(u^{N,k})^2u^{N,k+\frac 12}-(u^*)^2\frac {u^*+u^N_m}2\|_{\dot \H^{-1}}^2 \mathbb I_{\{\|u^{N,k}\|_{\dot \H^1}+\|u^N_m\|_{\dot \H^1} \le \frac 1\epsilon\}} \\
	&-\frac {h^2}4\<(\frac 12((u^*)^2+(u^N_m)^2)-1)\frac {u^*+u^N_m}2,\frac {u^*+u^N_m}2-u^{N,k+\frac 12}\>_{L^2}\mathbb I_{\{\|u^{N,k}\|_{\dot \H^1}+\|u^N_m\|_{\dot \H^1} > \frac 1\epsilon\}}\\
	\le&\frac {h^2}4 \|\frac {u^*+u^N_m}2-u^{N,k+\frac 12}\|_{L^2}^2\\
	&+\frac {h^2}8 \|(u^{N,k})^2u^{N,k+\frac 12}-(u^*)^2\frac {u^*+u^N_m}2\|_{\mathbb H^{-1}}^2 \mathbb I_{\{\|u^{N,k}\|_{\dot \H^1}+\|u^N_m\|_{\dot \H^1} \le \frac 1\epsilon\}} \\
	&+\frac {h^2}8\|(\frac 12((u^*)^2+(u^N_m)^2)-1)\frac {u^*+u^N_m}2\|^2_{\dot \H^{-1}}\mathbb I_{\{\|u^{N,k}\|_{\dot \H^1}+\|u^N_m\|_{\dot \H^1} > \frac 1\epsilon\}}.
	\end{align*}
	By applying the the Sobolev embedding  $ L^{\frac 65}\hookrightarrow\dot{\mathbb H}^{-\frac d3}$ and Young's inequality, we further have that for small $h\in (0,1),$
	\begin{align*}
	&\|\frac {u^*+u^N_m}2-u^{N,k+\frac 12}\|_{L^2}^2\\
	\le& C {h^2} \Big(\|u^{N,k}-u^*\|_{L^2}^2(\|u^{N,k}\|_{L^6}^2+\|u^*\|_{L^6}^2)(\|u^*\|_{L^6}^2+\|u^N_m\|_{L^6}^2)\\
	&+\|u^{N,k}\|_{L^6}^4\|u^{N,k+\frac 12}-\frac {u^*+u^N_m}2\|_{L^2}^2\Big) \mathbb I_{\{\|u^{N,k}\|_{\dot \H^1}+\|u^N_m\|_{\dot \H^1} \le \frac 1\epsilon\}}\\
	&+C {h^2}\Big(\|u^*\|_{L^2}^2+\|u^N_m\|_{L^2}^2+\|u^*\|_{L^6}^6+\|u^{N}_m\|_{L^6}^6\Big)\mathbb I_{\{\|u^{N,k}\|_{\dot \H^1}+\|u^N_m\|_{\dot \H^1} > \frac 1\epsilon\}}.
	\end{align*}
	{\color{black}Notice that} the energy preservation of $(u^*,v^*)$ leads to
	\begin{align*}
	\frac 12\|u^*\|_{\dot \H^1}^2+\frac 12\|v^*\|_{L^2}^2+\frac 14\|u^*\|_{L^4}^4
	&\le \frac 12\|u^{N}_m\|_{\dot \H^1}^2+\frac 12\|v^{N}_m\|_{L^2}^2+\frac 14\|u^{N}_m\|_{L^4}^4+C_1.
	\end{align*}
	Combining the above estimates, using Chebyshev's inequality and \eqref{pri-est-fix}, we conclude that 
	\begin{align*}
	&\|\frac {u^*+u^N_m}2-u^{N,k+\frac 12}\|_{L^2}^2\\
	\le& Ch^2\frac 1{\epsilon^8}\Big(\|u^{N,k}-u^*\|_{L^2}^2+\|\frac {u^*+u^N_m}2-u^{N,k+\frac 12}\|_{L^2}^2\Big)+C {h^2}\Big(\|u^*\|_{L^2}^2+\|u^N_m\|_{L^2}^2\\
	&+\|u^*\|_{L^6}^6+\|u^N_m\|_{L^6}^6\Big)\mathbb I_{\{\|u^{N,k}\|_{\dot \H^1}+\|u^N_m\|_{\dot \H^1} > \frac 1\epsilon\}}.
	\end{align*}
Letting $\frac {Ch^2\frac 4 {\epsilon^{8}}}{1-Ch^2\frac 1{\epsilon^8}}<\frac 12,$ we obtain that 
	\begin{align*}
	&\|\frac {u^*+u^N_m}2-u^{N,k+\frac 12}\|_{L^2}^2\\
	\le& \frac 12 \|\frac {u^*+u^N_m}2-u^{N,k-\frac 12}\|_{L^2}^2
	+\frac 18 \epsilon^8 (\|u^*\|_{L^2}^2+\|u^N_m\|_{L^2}^2+\|u^*\|_{L^6}^6+\|u^N_m\|_{L^6}^6) \\
	&\times \mathbb I_{\{\|u^{N,k}\|_{\dot \H^1}+\|u^N_m\|_{\dot \H^1} > \frac 1\epsilon\}}.
	\end{align*}
	
Applying the Gaussian type tail estimates for $(u_m^N,v_m^N)$ and $(u^*,v^*)$ which possess some exponential integrability (see e.g. the proof of [13, Corollary 4.1]), and using the Sobolev embedding theorem, the assumption $$\E\Big[\exp({V_1(u_m^N,v_m^N)}{e^{-\alpha t_m}})\Big]<\infty$$ in Appendix, as well as the definition of $V_1$, together with \eqref{pri-est-fix}, we obtain that for  $C'h(1+\frac 1{\epsilon^2})\le \widetilde C_1<\infty,$ $(1+C'h(1+\frac 1{\epsilon^2}))Ch(\frac 1 \epsilon+\frac 1 {\epsilon^3})\le \widetilde C_2 \ll \frac 1{\epsilon},$
 \begin{align*}
&\quad\mathbb P\Big(\{\|u^{N,k}\|_{\dot \H^1}+\|u^N_m\|_{\dot{\mathbb H}^{1}} \ge  \frac 1\epsilon\}\Big)\\
&\le \mathbb P\Big(\{\widetilde C_1\sqrt{\|u^N_{m}\|_{\dot \H^1}^2+\| v^{N}_m\|_{L^2}^2}+\widetilde C_2+\|u^N_m\|_{\dot{\mathbb H}^{1}} \ge \frac 1\epsilon\}\Big)\\
&\le \mathbb P\Big(\{(\widetilde C_1 +1)\sqrt{\|u^N_m\|_{\dot{\mathbb H}^{1}}^2 + \|v^N_{m}\|_{L^2}^2}\ge \frac 1\epsilon-\widetilde C_2 \}\Big)\\
&\le  \mathbb P\Big(\{\sqrt{\|u^N_m\|_{\dot{\mathbb H}^{1}}^2 + \|v^N_{m}\|_{L^2}^2} \ge \frac 1{\widetilde C_1 +1} (\frac 1 \epsilon-\widetilde C_2) \}\Big)\\
&\le C_{21}\exp(-\eta(\frac 1{\widetilde C_1 +1} (\frac 1 \epsilon-\widetilde C_2))^2),
 \end{align*}
 where $C_{21}=\E[\exp((\frac 12\|v_m^N\|_{L^2}^2+\frac 12 \|u_m^N\|^2_{\dot{\mathbb H}^1})e^{-\alpha {t_m}})]$ and $\eta=\frac 12 e^{\alpha t_m}. $

The above estimates, together with the fact that $V_1(u^*,v^*)=V_1(u^N_m,v^N_m)$, yield that there exist $\widetilde C>0$ and $\widetilde \eta>0$ depending on $T,Q,d,u^N_m,v^N_m$ such that
	\begin{align*}
	&\E\Big[\|\frac {u^*+u^N_m}2-u^{N,k+\frac 12}\|_{L^2}^2\Big]\\
	\le& \frac 12 \E\Big[\|\frac {u^*+u^N_m}2-u^{N,k-\frac 12}\|_{L^2}^2\Big]
	+\widetilde C \epsilon^8 \Big(\mathbb P\Big(\{\|u^{N,k}\|_{\dot \H^1}+\|u^N_m\|_{\dot \H^1} > \frac 1\epsilon\}\Big)\Big)^{\frac 12}\\
	\le& \frac 12 \mathbb E\Big[\|\frac {u^*+u^N_m}2-u^{N,k-\frac 12}\|_{L^2}^2\Big]
	+\widetilde C\epsilon^8  C_{21}^{\frac 12}\exp(-\eta(\frac 1{2(\widetilde C_1 +1)} (\frac 1 \epsilon-\widetilde C_2))^2)\\
	\le& \frac 12 \mathbb E\Big[\|\frac {u^*+u^N_m}2-u^{N,k-\frac 12}\|_{L^2}^2\Big]
+\widetilde C\epsilon^8\exp(- \widetilde \eta\frac 1{\epsilon^2}),
	\end{align*}
where $\widetilde C$ in the last equality differs from previous one. 		
Then by iteration arguments, we achieve that 
	\begin{align*}
	&\E\Big[\|\frac {u^*+u^N_m}2-u^{N,k+\frac 12}\|_{L^2}^2\Big]\le (\frac 12)^{k+1}\frac 14\E\Big[\|u^*-u^N_m\|_{L^2}^2\Big] + 2\widetilde C\epsilon^8 \exp(- \eta\frac 1{\epsilon^2}).
	\end{align*}
	Since $u^*$ is the numerical solution of scheme \eqref{sche;sbes;det}, the fact that $V_1(u^*,v^*)=V_1(u^N_m,v^N_m)$ and the moment boundedness of $V_1(u^N_m,v^N_m)$ yield that
	\begin{align*}
	\E\Big[\|u^*-u^N_m\|_{L^2}^2\Big]=\frac 14 h^2\E\Big[\|v^N_{m}+v^*\|_{L^2}^2\Big].
	\end{align*}
	By taking $\frac {Ch^2\frac 4 {\epsilon^{8}}}{1-Ch^2\frac 1{\epsilon^8}}<\frac 12,$ $C'h(1+\frac 1{\epsilon^2})\le \widetilde C_1<\infty,$ $(1+C'h(1+\frac 1{\epsilon^2}))Ch(\frac 1 \epsilon+\frac 1 {\epsilon^3})\le \widetilde C_2 \ll \frac 1{\epsilon},$
	the above estimates give that there exist $C,\eta>0$ depending on $Q,T,u_m^N,v_m^N,d$ such that
	\begin{align*}
	\E\Big[\|u^*-u^{N,k+1}\|_{L^2}^2\Big]&\le \max\Big(\frac 14\E\Big[\|u^*-u^N_m\|_{L^2}^2\Big],2C'\epsilon^8\Big)\Big((\frac 12)^{k+1}+\exp(-\eta \frac 1{\epsilon^2})\Big)\\
	&\le C\max(h^2,\epsilon^8)\Big((\frac 12)^{k+1}+\exp(-\eta \frac 1{\epsilon^2})\Big).
	\end{align*} 
	One can obtain the estimate of $\mathbb E\Big[\|v^*-v^{N,k+1}\|^2_{\dot{\mathbb H}^{-1}}\Big]$ by similar arguments. \hfill$\square$

\bibliography{references}

\def\cprime{$'$} \def\cprime{$'$}
\begin{thebibliography}{10}

\bibitem{AMS94}
S.~Aida, T.~Masuda, and I.~Shigekawa.
\newblock Logarithmic {S}obolev inequalities and exponential integrability.
\newblock {\em J. Funct. Anal.}, 126(1):83--101, 1994.

\bibitem{ACLW16}
R.~Anton, D.~Cohen, S.~Larsson, and X.~Wang.
\newblock Full discretization of semilinear stochastic wave equations driven by
  multiplicative noise.
\newblock {\em SIAM J. Numer. Anal.}, 54(2):1093--1119, 2016.

\bibitem{BG99}
S.~G. Bobkov and F.~G\"{o}tze.
\newblock Exponential integrability and transportation cost related to
  logarithmic {S}obolev inequalities.
\newblock {\em J. Funct. Anal.}, 163(1):1--28, 1999.

\bibitem{BCH18}
C.~E. Br\'{e}hier, J.~Cui, and J.~Hong.
\newblock Strong convergence rates of semidiscrete splitting approximations for
  the stochastic {A}llen-{C}ahn equation.
\newblock {\em IMA J. Numer. Anal.}, 39(4):2096--2134, 2019.

\bibitem{BT79}
P.~Brenner and V.~Thom\'ee.
\newblock On rational approximations of semigroups.
\newblock {\em SIAM J. Numer. Anal.}, 16(4):683--694, 1979.

\bibitem{C02}
P.~Chow.
\newblock Stochastic wave equations with polynomial nonlinearity.
\newblock {\em Ann. Appl. Probab.}, 12(1):361--381, 2002.

\bibitem{Chow06}
P.~Chow.
\newblock Asymptotics of solutions to semilinear stochastic wave equations.
\newblock {\em Ann. Appl. Probab.}, 16(2):757--789, 2006.

\bibitem{CLS13}
D.~Cohen, S.~Larsson, and M.~Sigg.
\newblock A trigonometric method for the linear stochastic wave equation.
\newblock {\em SIAM J. Numer. Anal.}, 51(1):204--222, 2013.

\bibitem{CHJ13}
S.~Cox, M.~Hutzenthaler, and A.~Jentzen.
\newblock Local {L}ipschitz continuity in the initial value and strong
  completeness for nonlinear stochastic differential equations.
\newblock {\em arXiv:1309.5595}.

\bibitem{CH17}
J.~Cui and J.~Hong.
\newblock Analysis of a splitting scheme for damped stochastic nonlinear
  {S}chr\"{o}dinger equation with multiplicative noise.
\newblock {\em SIAM J. Numer. Anal.}, 56(4):2045--2069, 2018.

\bibitem{CH18}
J.~Cui and J.~Hong.
\newblock Strong and weak convergence rates of finite element method for
  stochastic partial differential equation with non-globally {L}ipschitz
  coefficients.
\newblock {\em SIAM J. Numer. Anal.}, 57(4):1815--1841, 2019.

\bibitem{CHL16b}
J.~Cui, J.~Hong, and Z.~Liu.
\newblock Strong convergence rate of finite difference approximations for
  stochastic cubic {S}chr\"odinger equations.
\newblock {\em J. Differential Equations}, 263(7):3687--3713, 2017.

\bibitem{CHLZ17b}
J.~Cui, J.~Hong, Z.~Liu, and W.~Zhou.
\newblock Strong convergence rate of splitting schemes for stochastic nonlinear
  {S}chr\"{o}dinger equations.
\newblock {\em J. Differential Equations}, 266(9):5625--5663, 2019.

\bibitem{CHS18b}
J.~Cui, J.~Hong, and L.~Sun.
\newblock On global existence and blow-up for damped stochastic nonlinear
  {S}chr\"{o}dinger equation.
\newblock {\em Discrete Contin. Dyn. Syst. Ser. B}, 24(12):6837--6854, 2019.

\bibitem{DS09}
R.~C. Dalang and M.~Sanz-Sol\'{e}.
\newblock H\"{o}lder-{S}obolev regularity of the solution to the stochastic
  wave equation in dimension three.
\newblock {\em Mem. Amer. Math. Soc.}, 199(931):vi+70, 2009.

\bibitem{Drag03}
S.~S. Dragomir.
\newblock {\em Some {G}ronwall type inequalities and applications}.
\newblock Nova Science Publishers, Inc., Hauppauge, NY, 2003.

\bibitem{Fur01}
D.~Furihata.
\newblock Finite-difference schemes for nonlinear wave equation that inherit
  energy conservation property.
\newblock {\em J. Comput. Appl. Math.}, 134(1-2):37--57, 2001.

\bibitem{HK98}
Y.~Hu and G.~Kallianpur.
\newblock Exponential integrability and application to stochastic quantization.
\newblock {\em Appl. Math. Optim.}, 37(3):295--353, 1998.

\bibitem{HJ14}
M.~Hutzenthaler and A.~Jentzen.
\newblock On a perturbation theory and on strong convergence rates for
  stochastic ordinary and partial differential equations with nonglobally
  monotone coefficients.
\newblock {\em Ann. Probab.}, 48(1):53--93, 2020.

\bibitem{HJW18}
M.~Hutzenthaler, A.~Jentzen, and X.~Wang.
\newblock Exponential integrability properties of numerical approximation
  processes for nonlinear stochastic differential equations.
\newblock {\em Math. Comp.}, 87(311):1353--1413, 2018.

\bibitem{JJT15}
L.~Jacobe~de Naurois, A.~Jentzen, and T.~Welti.
\newblock Weak convergence rates for spatial spectral {G}alerkin approximations
  of semilinear stochastic wave equations with multiplicative noise.
\newblock {\em arXiv:1508.05168, To appear in Appl. Math. Optim.}

\bibitem{KLL13}
M.~Kov\'{a}cs, S.~Larsson, and F.~Lindgren.
\newblock Weak convergence of finite element approximations of linear
  stochastic evolution equations with additive noise {II}. {F}ully discrete
  schemes.
\newblock {\em BIT}, 53(2):497--525, 2013.

\bibitem{KLS10}
M.~Kov\'{a}cs, S.~Larsson, and F.~Saedpanah.
\newblock Finite element approximation of the linear stochastic wave equation
  with additive noise.
\newblock {\em SIAM J. Numer. Anal.}, 48(2):408--427, 2010.

\bibitem{MSH02}
J.~C. Mattingly, A.~M. Stuart, and D.~J. Higham.
\newblock Ergodicity for {SDE}s and approximations: locally {L}ipschitz vector
  fields and degenerate noise.
\newblock {\em Stochastic Process. Appl.}, 101(2):185--232, 2002.

\bibitem{PZ00}
S.~Peszat and J.~Zabczyk.
\newblock Nonlinear stochastic wave and heat equations.
\newblock {\em Probab. Theory Related Fields}, 116(3):421--443, 2000.

\bibitem{S08}
H.~Schurz.
\newblock Analysis and discretization of semi-linear stochastic wave equations
  with cubic nonlinearity and additive space-time noise.
\newblock {\em Discrete Contin. Dyn. Syst. Ser. S}, 1(2):353--363, 2008.

\bibitem{WGT14}
X.~Wang, S.~Gan, and J.~Tang.
\newblock Higher order strong approximations of semilinear stochastic wave
  equation with additive space-time white noise.
\newblock {\em SIAM J. Sci. Comput.}, 36(6):A2611--A2632, 2014.

\end{thebibliography}
\bibliographystyle{plain}
\end{document}